\documentclass[11pt]{article}%
\usepackage{amssymb,amsmath,amsfonts,amsthm,array,bm,bbm,color}%
\usepackage[style=numeric,backend=biber,maxbibnames=10,isbn=false,url=false]{biblatex}
\usepackage{mathrsfs}%
\usepackage{graphicx}%
\usepackage{enumerate} 
\providecommand{\U}[1]{\protect\rule{.1in}{.1in}}

\usepackage[pdfstartview=FitH]{hyperref}
\usepackage{xcolor} 

\numberwithin{equation}{section}

\newtheorem{theorem}{Theorem}[section]
\newtheorem{lemma}[theorem]{Lemma}
\newtheorem{corollary}[theorem]{Corollary}
\newtheorem{proposition}[theorem]{Proposition}
\newtheorem{remark}[theorem]{Remark}

\newtheorem{definition}[theorem]{Definition}

\def\<{\langle}
\def\>{\rangle}
\def\d{{\rm d}}

\def\E{\mathbb{E}}

\def\N{\mathbb{N}}
\def\P{\mathbb{P}}
\def\R{\mathbb{R}}

\def\W{\mathbb{W}}
\def\Z{\mathbb{Z}}

\def\cA{\mathcal{A}}
\def\cF{\mathcal{F}}
\def\cI{\mathcal{I}}

\def\cM{\mathcal{M}}
\def\cN{\mathcal{N}}
\def\cY{\mathcal{Y}}

\def\sC{\mathscr{C}}

\begin{document}

\title{Large deviation principle for the stationary solutions of stochastic functional differential equations with infinite delay}
\author{Yong Liu, \footnote{Email: liuyong@math.pku.edu.cn} \quad Bin Tang \footnote{Email: tangbin@math.pku.edu.cn, Corresponding author}
\bigskip \\
{\small LMAM, School of Mathematical Sciences, Peking University,} \\
{\small Beijing 100871, China}}

\maketitle

\vspace{-15pt}

\begin{abstract}
 We investigate the large deviation principle (LDP) of the stationary solutions of stochastic functional differential equations (SFDEs) with infinite delay under small random perturbations. First, we demonstrate the existence and uniqueness of the corresponding stationary solutions. Second, by the weak convergence approach, we show the uniform large deviation principle for the solution maps, and then prove the LDP for stationary solutions. Furthermore, we obtain the LDP for invariant measures of SFDEs through the LDP for stationary solutions and the contraction principle.
\end{abstract}



\section{Introduction} \label{sec: introduction}

In the real world, dynamic systems usually depend on their historical states. Moreover, uncertainties within these systems have prompted research into stochastic functional differential equations (SFDEs) with delay, especially with infinite delay. In this paper, we consider SFDEs perturbed by small noise with infinite delay starting from the initial segment $\xi \in C_r$ at $t_0 \in \R$:
\begin{equation} \label{eq: truncated SFDEs epsilon}
 \begin{cases}
 \d Y^{\epsilon}(t; t_0, \xi)=b(Y_{t;t_0,\xi}^{\epsilon}) \, \d t+ \sqrt{\epsilon} \sigma(Y_{t;t_0,\xi}^{\epsilon}) \, \d W(t), \quad & \mbox{if} \quad t > t_0; \\
 Y^{\epsilon}(t; t_0, \xi)= \xi(t-t_0), \quad & \mbox{if} \quad t \leq t_0; \\
 Y_{t;t_0,\xi}^{\epsilon} := \left\{ Y^{\epsilon}(t+\tau;t_0,\xi); \, \tau \in (-\infty,0] \right\} \in C_r,
 \end{cases}
\end{equation}
where $\epsilon \in (0,1)$, $W(t)$ is an $m$-dimensional double-side Wiener process, $C_r$ is the phase space defined as \eqref{def: C_r}, $b:C_{r} \rightarrow \R^{d}$ and $\sigma:C_{r} \rightarrow \R^{d\times m}$ are continuous functionals. We will show the large deviation principle (LDP) for the solution maps $\{ Y_{\cdot;t_0,\xi}^{\epsilon}\}_{\epsilon>0}$ uniform over the compact sets of the path space $\sC_r$ (see \eqref{def: path space}). Eq. \eqref{eq: truncated SFDEs epsilon} induces the corresponding dynamic system $(U^{\epsilon},\theta)$ (see \eqref{U(t,theta)}), which is a crude cocycle. Furthermore, we will show the existence, pathwise uniqueness, and LDP for the stationary solutions $\{ \mathcal{Y}^{\epsilon} \}_{\epsilon>0}$ for the crude cocycle $(U^{\epsilon},\theta)$, as defined precisely by \eqref{identity: stationary-solutions} in Definition \ref{def: stationary solution}. It implies the LDP for the invariant measure $\{ \nu^{\epsilon} \}_{\epsilon > 0}$ of Eq. \eqref{eq: truncated SFDEs epsilon} by using the contraction principle.

Research on deterministic and stochastic functional differential equations with infinite delay originated in the 1970s, finding applications in various fields such as biology, epidemiology, mechanics, neural networks, etc., see \cite{Caraballo_Attractors_2007, Hale_Functional_1974, Hino_Functionaldifferential_1991, Kolmanovskii_Stability_1986, Li_Stability_2010, McShane_Stochastic_1969, Ren_Remarks_2008, Rihan_Delay_2021, Wu_Stochastic_2017}. Wu et al. \cite{Wu_Stochastic_2017} show the existence, uniqueness, Markov property, and ergodicity of Eq. \eqref{eq: truncated SFDEs epsilon} under dissipative and local Lipschitz conditions, which we will adopt in our analysis. They also provided the existence and uniqueness of the invariant measure of Markov process $Y_{t;t_0,\xi}^{\epsilon}$ in \cite[Theorem 5.1]{Wu_Stochastic_2017}. Wang et al. \cite{Wang_Stochastic_2022} further extended these results to include conditions with non-Lipschitz coefficients. Notably, the Markov property holds only for the solution map $Y_{t;t_0,\xi}^{\epsilon}$ but not for the solution $Y^{\epsilon}({t;t_0,\xi})$ (refer to \cite[Section 4.2]{Wu_Stochastic_2017}). We therefore regard the solution map as the ``true'' solution and define its path space as $\sC_r$ (see \eqref{def: path space}). The path space $\sC_r $ is a subspace of continuity function space $C(\R;C_r)$ with compact-open topology, and it will be convenient when discussing the stationary solution. Some studies have focused on stationary solution of SFDEs with finite delay \cite{Caraballo_Existence_2007, Jiang_Global_2023}, as well as on stationary solutions SFDEs with infinite delay in the finite dimensional distributional sense \cite{Bakhtin_Existence_2002, Bakhtin_Stationary_2005, Ito_stationary_1964}. Using the remote start method or called pull-back approach (see \cite{Flandoli_Weak_1999, Liu_Representation_2009}), we construct the pathwise unique stationary solution for SFDEs with infinite delay. To our knowledge, this paper is the first to investigate the strong stationary solution of SFDEs with infinite delay (see Definition \ref{def: stationary solution}, Theorem \ref{thm: stationary solution}, and the subsequent discussion). The existence and pathwise uniqueness of the strong stationary solutions shows the ``one force one solution'' principle and is necessary in the proof of the LDP for stationary solutions via the weak convergence approach.

The large deviation principle (LDP) is a fundamental theory that calculates the probability of rare events and provides the exponential tail estimates of some probability distributions. In 1966, Varadhan \cite{Varadhan_Asymptotic_1966} gave a modern formulation of LDP (see Definition \ref{def: large deviation}). 
From the 1960s onward, Freidlin and Wentzell \cite{Freidlin_Random_2012} developed a framework for investigating the long-time behavior of dynamical systems perturbed by a small random noise through LDP. 
Dupuis, Ellis, Bou\'e and Budhiraja et al. \cite{Boue_variational_1998, Budhiraja_variational_2000, Budhiraja_Analysis_2019, Budhiraja_Large_2008, Dupuis_Weak_1997} established the weak convergence approach, which provides another powerful method to explore the LDP through a variational representation. For SFDEs with finite delay, various results on LDP have been established, including small noise LDP \cite{Jin_Large_2022, Mo_Large_2013, Mohammed_Large_2006}, small noise uniform LDP and exit time asymptotics \cite{Lipshutz_Exit_2018}, as well as small time LDP \cite{Yang_small_2023}. For SFDEs with infinite delay, Gopal and Suvinthra \cite{Gopal_Large_2020} found small noise LDP with Lipschitz conditions and driven by both additive and multiplicative noise; Wang et al. \cite{Wang_Large_2024} used the Freidlin-Wentzell's theory to derive small noise LDP for systems with regime-switching and local Lipschitz conditions. Different from the above LDP results in a finite time interval, to explore the asymptotic behavior in the whole time interval, we further establish the small noise {\emph{uniform LDP}} for SFDEs with infinite delay in $ \sC_r $, specifically for the distributions of continuous mapping $ \{ t \mapsto Y^{\epsilon}_{t; t_0, \xi}; \, t \in \mathbb{R} \} $. It is also an intermediate step in proving the LDP for stationary solutions. 

To describe the transitions between stationary behaviors of dynamic systems under small noise, Freidlin \cite{Freidlin_Random_1988} introduced the concept of LDP for invariant measures. By constructing an auxiliary function known as the quasi-potential, Freidlin, Cerrai, Brzeźniak et al. \cite{Brzezniak_Large_2017,Cerrai_Large_2005,Cerrai_Large_2022,Freidlin_Random_1988} demonstrated the LDP for the invariant measures of the reaction-diffusion equation and the 2D Navier-Stokes equation. However, constructing the quasi-potential for SFDEs with infinite delays is particularly challenging. This difficulty arises because the dynamic behaviors of SFDEs depend on all historical states, and their invariant measures are supported in an infinite-dimensional space. Consequently, applying the quasi-potential method to prove the LDP for the invariant measures of SFDEs is quite complex. 

The stationary solution is a crucial topic in the study of SFDEs with infinite delay, highlighting the irreversible impact of non-vanishing historical dependence on long-term stability. Therefore, investigating the LDP for the stationary solution of SFDEs with infinite delay is a natural inquiry, which reveals how these historically dependent perturbations affect long-term stability. Recently, Gao et al. in \cite[Example 4]{Gao_Large_2022} show that the LDP for stationary solutions offers more precise insights into the long-term behavior of dynamical systems compared to the LDP for invariant measures. Using the weak convergence method, Gao et al. \cite{Gao_Large_2022} derived the LDP for stationary solutions of a class of SPDEs, where the stationary solution can be represented as a mild solution and naturally be represented as a Borel measurable map $\mathcal{G}^{\epsilon}(W)$. However, for SFDEs \eqref{eq: truncated SFDEs epsilon}, their stationary solutions cannot generally be represented as mild solutions. To establish the existence of a Borel measurable map $\mathcal{G}^{\epsilon}$ satisfying $\mathcal{G}^{\epsilon}(W)$ is an $\mathcal{F}_{t}$-adapted stationary solution to SFDEs \eqref{eq: truncated SFDEs epsilon}, we need construct the pathwise unique stationary solutions (Theorem \ref{thm: stationary solution}) and demonstrate a modified Yamada--Watanabe--Engelbert Theorem in the whole time interval (Lemma \ref{lem: Yamada-Watanabe}). It need to note that the Bou\'e--Dupuis formula over infinite time interval (see \cite[Theorem 9]{Lehec_Representation_2013} and Lemma \ref{lem: variational representation}) serves as the foundation for applying the weak convergence method in proving the LDP of the stationary solutions.

The LDP for invariant measures derives directly from the LDP for stationary solutions via the contraction principle. The rate functions from both LDPs are equivalent due to the uniqueness of the rate function. This approach not only elucidates the quasi-potential from the perspective of the LDP for stationary solutions but also provides an alternative method for deriving the rate function of invariant measures.

The proof of LDP for stationary solutions of SFDEs \eqref{eq: truncated SFDEs epsilon} is divided as the following steps. Firstly, we solved the preconditions for applying the weak convergence method in Theorem \ref{thm: stationary solution} and Lemma \ref{lem: variational representation}. Secondly, we proved the uniform LDP for SFDEs \eqref{eq: truncated SFDEs epsilon} in $ \sC_r $. Thirdly, using the stochastic Gr\"onwall inequality (Lemma \ref{lem: Stochastic Gronwall}), we obtained the uniform convergence of the controlled solution maps $\mathcal{Y}^{\epsilon,v_{\epsilon}}_{-n,\xi}:=\mathcal{G}_{t_0; \xi}^{\epsilon} \big(W(\cdot)+\frac{1}{\sqrt{\epsilon}} \int_{0}^{\cdot} v_{\epsilon}(s) \, \d s\big) $, see details in Lemma \ref{lem: truncated-controlled-equation} and Lemma \ref{lem: uniform-controlled-equation}. Finally, using the uniform LDP and the uniform convergence, we verified the sufficient conditions given by the weak convergence method and established the LDP for stationary solutions of SFDEs \eqref{eq: truncated SFDEs epsilon}, see Theorem \ref{thm: LDP-stationary solution}.

The remainder of this paper is organized as follows. Section \ref{sec: Preliminary} introduces the necessary notations and assumptions, and presents the main results. Section \ref{sec: SFDE and Stationary solution} demonstrates the existence and uniqueness of strong solutions of SFDEs \eqref{eq: truncated SFDEs epsilon} and the corresponding stationary solutions. Section \ref{sec: Large Deviation} provides the detailed proof of the uniform LDP for the solution maps of the SFDEs \eqref{eq: truncated SFDEs epsilon} and the LDP for stationary solutions. Appendix \ref{sec: analysis tools} presents the primary analysis tools.

\section{Preliminary and main results} \label{sec: Preliminary}


We first introduce some basic notation and assumptions. Let $(\Omega,{\mathcal{F}},\mathbb{P},({\mathcal{F}}_{t})_{t>-\infty})$ be a filtered probability space. The filtration $\cF_t$ is defined by 
\begin{equation*}
 \cF_t := \cap_{s>t} \, \sigma \left( \{W(t_2)-W(t_1): -\infty < t_1<t_2 \leq s \} , \, \cN \right),
\end{equation*}
where $\cN$ is the set of all $\P$-null sets. $W(t)$ is the $m$-dimensional double-side Wiener process on this probability space, whose path lies on the space:
\begin{equation*}
 \W=\Big\{ \theta\in C(\mathbb{R}; \, \R^m); \; \theta(0)=0\Big\},\quad
 \|\theta\|_{\W}:= \sum_{n=1}^{+\infty} \frac{ (\sup_{|t| \leq n}|\theta(t)|) \wedge 1}{2^{n}}.
\end{equation*}

Let $C((-\infty,0];{\R}^{d})$ be the family of continuous functions from $(-\infty,0]$ to $\R^d$. For any given $r>0$, the phase space with the fading memory is defined as follows:
\begin{equation} \label{def: C_r}
 C_{r}:=\left\{\phi\in C((-\infty,0];{\R}^{d}): \; \lim_{\tau \rightarrow - \infty} e^{r \, \tau} \phi(\tau) \; \mbox{exists in} \; \R^d \right\} 
\end{equation}
equipped with the norm $\|\phi\|_{r} = \sup_{-\infty<\tau\leq 0}e^{r \, \tau}|\phi(\tau)|$. $(C_r, \| \cdot \|_{r})$ is a Polish space, cf. \cite{Hino_Functionaldifferential_1991}. Additional properties of space $ C_r$ that we will use are listed in Proposition \ref{prop: cr}. To consider the path of solution maps, we defined the path space $ \sC_r$ as
\begin{equation} 
  \sC_{r}:=\Big\{\Phi \in C(\R;C_r): \, \Phi(t)(\tau) = \Phi(t+\tau)(0), \quad \forall \, t \in \R, \, \forall \, \tau \in (-\infty,0] \Big\} \label{def: path space}
\end{equation}
with the metric defined as $d_{r}(\Phi,\Psi ):= \big( \sum_{n=1}^{+\infty} \frac{\|\Phi(n)-\Psi(n)\|_{r}^2 \wedge 1 }{2^{n} } \big)^{1/2}$ for any $\Phi, \Psi \in \sC_r$, where $C(\R;C_r)$ is the space of all continuous functions from $\R$ to $C_r$. According to Proposition \ref{prop: cr} ii), we have $\sup_{|t| \leq n} \| \Phi(t) \|_{r}^2 \leq e^{4nr} \| \Phi(n) \|_{r}^2$. Then for any $\Phi, \Psi \in \sC_r$, we have
\begin{equation*}
 \sum_{n=1}^{+\infty} \frac{\sup_{|t| \leq n} \| \Phi(t)-\Psi(t) \|_{r}^2 \wedge 1 }{ 2^{n} e^{4nr} } \leq |d_{r}(\Phi,\Psi)|^2 \leq \sum_{n=1}^{+\infty} \frac{\sup_{|t| \leq n} \| \Phi(t) -\Psi(t)\|_{r}^2 \wedge 1 }{ 2^{n} }.
\end{equation*}
Thus, $\sC_r$ equipped with metric $d_r$ is a closed subspace of $C(\R;C_r)$ with compact-open topology. Since $C_r$ is a Polish space, $\sC_{r}$ with metric $d_r$ is also a Polish space. For any $p \geq 1$, we define $L^{p}(\Omega,\sC_r)$ as the space of all random functions lie on $\sC_r$ with the metric $ |\E |d_{r} (\cdot,\cdot)|^{p}|^{1/p} $.

We denote $\mathcal{M}_{0}$ as the set of probability measures on $(-\infty,0]$, namely, for any $\mu\in{\cM}_{0}$, $\int_{-\infty}^{0}\mu(\d \tau)=1.$ For any $\kappa>0$, let us further define the measure space $\cM_{\kappa}$ as follows: 
\begin{equation} \label{def: mu}
 {\cM}_{\kappa}:=\Big\{\mu\in{\cM}_{0};\, \mu^{(\kappa)}:=\int_{-\infty}^{0} e^{-\kappa\tau}\mu(\d \tau)<\infty\Big\}.
\end{equation}
To construct stationary solutions and establish the LDP for SFDEs \eqref{eq: truncated SFDEs epsilon}, we suppose continuous functionals $b:C_{r} \rightarrow \R^{d}$ and $\sigma:C_{r} \rightarrow \R^{d\times m}$ satisfy the following assumptions: 
\begin{enumerate}[(H1)] 
 \item (The local Lipschitz condition). For any $k>0$, there exists a constant $c_{k}$ such that
 \begin{equation*}
  |b(\varphi)-b(\phi)| \leq c_{k}\|\varphi-\phi\|_{r} 
 \end{equation*}
 for any $\varphi,\phi\in C_{r}$ and $\|\varphi\|_{r} \lor \|\phi\|_{r}\leq k$. 
 
 \item There exist constants $\lambda_{1},\lambda_{2}\geq 0$ and $\mu_{1} \in{\mathcal{M}}_{2r}$ such that for any $\varphi, \phi \in C_r$,
 \begin{equation*}
  \big(\varphi(0)-\phi(0) \big)^{\prime} \big(b(\varphi)-b(\phi) \big) \leq-\lambda_{1}|\varphi(0)-\phi(0)|^{2}+\lambda_{2}\int_{-\infty}^{0}|\varphi(\tau)-\phi(\tau)|^{2} \mu_{1}(\d \tau) .
 \end{equation*}

 \item There exist constants $\lambda_{3}\geq 0$ and $\mu_{2} \in{\mathcal{M}}_{2r}$ such that for any $\varphi, \phi \in C_r$,
 \begin{equation*}
  \big|\sigma(\varphi)-\sigma(\phi) \big|^2 \leq \lambda_{3}\int_{-\infty}^{0}|\varphi(\tau)-\phi(\tau)|^{2} \mu_{2}(\d \tau) .
 \end{equation*}
\end{enumerate}
The assumption (H3) implies that $\sigma$ is global Lipschitz continuous, cf. \cite[Remark 3.3]{Wu_Stochastic_2017}. We split the condition (H2) given in \cite{Wu_Stochastic_2017} as the assumptions (H2) and (H3) above. The combined assumptions (H1) and (H3) above is stronger than (H1) given in \cite{Wu_Stochastic_2017}. Thus, assumptions (H1), (H2) and (H3) is equivalence the assumptions (H1) and (H2) given in \cite{Wu_Stochastic_2017}. For the existence of the stationary solutions, we will further assume $2\lambda_{1}>2\lambda_{2}\mu_{1}^{(2r)}+\lambda_{3}\mu_{2}^{(2r)}$.

Under the above assumptions, Wu et al. \cite{Wu_Stochastic_2017} give the existence and pathwise uniqueness of the global solution $Y^{\epsilon}(t;t_0,\xi)$ and solution maps $Y^{\epsilon}_{t;t_0,\xi}$. In Lemma \ref{lem: truncated-measurable-map}, we further establish the existence of Borel measurable map $\mathcal{G}^{\epsilon}_{t_0}: \W \times C_r \rightarrow \sC_r$ such that $\mathcal{G}^{\epsilon}_{t_0}(W,\xi) (t,\omega)$ is $\cF_{t}$-adapted and satisfies Eq. \eqref{eq: truncated SFDEs epsilon} with initial data $\xi \in C_r$. Therefore, for any fixed $t_0 \in \R$, we can define the demapping $U_{t_0}^{\epsilon}$ from $\R_{+} \times C_r \times \Omega$ to $C_r$ as 
\begin{align}
 U_{t_0}^{\epsilon}: (t,\xi,\omega) \mapsto \mathcal{G}_{t_0}^{\epsilon}(W,\xi)(t+t_0, \omega) & =Y^{\epsilon}_{t+t_0;t_0,\xi}(\cdot,\omega) \nonumber \\
 &  =\{ Y^{\epsilon}(t+t_0+\tau; \, 0,\xi)(\omega); \; \tau \in (-\infty,0] \}. \label{U(t,theta)}
\end{align}

The standard $\mathbb{P}$-preserving ergodic Wiener shift $\theta:\mathbb{R}\times \Omega\rightarrow \Omega$ is defined by $\theta(t,\omega(s)):=\omega(t+s)-\omega(s)$ for any $s,t\in\mathbb{R}$. Moreover, denote $\theta(t,\omega(0))$ as $\theta(t,\omega)$. Using the stationary increments of the Wiener process, it holds that for any $t_0 \in \R$, $\P$-a.s.
\begin{equation} \label{ident: trans}
 U_{t_0}^{\epsilon}\left(t,\xi, \omega\right)=Y^{\epsilon}_{t+t_0;t_0,\xi}(\cdot,\omega)=Y^{\epsilon}_{t;0,\xi}(\cdot,\theta(t_0,\omega))=U_0^{\epsilon}\left(t,\xi, \theta\left(t_0, \omega\right)\right), \quad \forall \, \xi \in C_r,  \, \forall \, {t \geq 0}. 
\end{equation} 
For any $t_0 \in \R$, since the solution map $Y^{\epsilon}_{t;t_0,\xi}$ is a $C_r$-valued strong homogeneous Markov process, as proved in \cite[Theorem 4.2]{Wu_Stochastic_2017}, we obtain that for any $s,t \geq 0$, $\P$-a.s.
\begin{equation*}
  U^{\epsilon}_{t_0}(t+s,\xi,\omega) = U_{t_0+s}^{\epsilon}(t,U^{\epsilon}_{t_0}(s,\xi,\omega),\omega)= U_{t_0}^{\epsilon}(t,U^{\epsilon}_{t_0}(s,\xi,\omega),\theta(s,\omega)).
\end{equation*}
Then for any $t_0 \in \R$, the random dynamical system $(U^{\epsilon}_{t_0},\theta)$ is a crude cocycle defined in \cite[Defintion 1.1.2]{Mohammed_stable_2008}, which means it satisfies the following conditions:
\begin{itemize}
 \item for any $s,t \geq 0$, $\mathbb{P}$-a.s., $U_{t_0}^{\epsilon} \left(t+s, \cdot, \omega\right)=U_{t_0}^{\epsilon} \left(t, U_{t_0}^{\epsilon} \left(s, \cdot, \omega\right), \theta\left(s, \omega\right)\right)$; 
 \item $U_{t_0}^{\epsilon} (0, \xi, \omega)=\xi$ for all $\xi \in C_r$, $\omega \in \Omega$. 
\end{itemize}
We call $(U^{\epsilon},\theta)$ is a crude cocycle, if the above equality \eqref{ident: trans} holds and for every $t_0 \in \R$, the random system $(U^{\epsilon}_{t_0},\theta)$ is a crude cocycle. This definition removes the dependence on the starting time and allows stationary solutions to be defined on the entire real line. 

Stationary solution is a fundamental concept in studying the long time behavior of stochastic dynamical systems (see \cite{Bakhtin_Existence_2002, Bakhtin_Stationary_2005, Ito_stationary_1964, Mohammed_stable_2008}). For crude cocycle $(U^{\epsilon},\theta)$, it is defined as follows. 
\begin{definition} \label{def: stationary solution}
 An $\mathcal{F}_t$ progressively measurable process $\mathcal{Y}^{\epsilon}: \R \times \Omega \rightarrow C_r$ is said to be a stationary solution for the crude cocycle $(U^{\epsilon}, \theta)$ if for any fixed $s \in \R$, $\P$-a.s.
	\begin{equation} \label{identity: stationary-solutions}
		U_0^{\epsilon}(t, \mathcal{Y}^{\epsilon}(s,\omega), \theta(s, \omega))=\mathcal{Y}^{\epsilon}(t+s,\omega)=\mathcal{Y}^{\epsilon}(t,\theta(s, \omega)), \quad \forall t \, \geq 0,
	\end{equation}
  and $\cY^{\epsilon}(\cdot,\omega): \R \rightarrow C_r$ lies on $\sC_r$ for any $\omega \in \Omega$.
\end{definition}
\begin{remark}
  For any $s \in \R$, the equality \eqref{ident: trans} and the first equality in \eqref{identity: stationary-solutions} show
  \begin{equation*}
    U_{s}^{\epsilon}(t, \mathcal{Y}^{\epsilon}(s,\omega), \omega)=U_0^{\epsilon}(t, \mathcal{Y}^{\epsilon}(s,\omega), \theta(s, \omega))=\mathcal{Y}^{\epsilon}(t+s,\omega), \quad \forall \, t \geq 0,
  \end{equation*}
which means that $\mathcal{Y}^{\epsilon}(t,\omega)$ is the solution of Eq. \eqref{eq: truncated SFDEs epsilon} with initial data $\mathcal{Y}^{\epsilon}(s, \omega)$ and any starting time $s \in \R$. The second equality of \eqref{identity: stationary-solutions} shows that the value of stationary solution at different times differ only by a Wiener shift in the Wiener space. 
\end{remark}

Let $ \nu^{\epsilon}:= \P \circ (\mathcal{Y}^{\epsilon}(0))^{-1}$. From the definition of stationary solution, we know that 
$$\nu^{\epsilon}(A)=\int_{C_r} \P (Y^{\epsilon}_{t;0,\phi} \in A) \, \nu^{\epsilon}(\d \phi), $$ 
for any Borel measurable set $A \in \mathcal{B}(C_r)$ and any $t\geq 0$. Thus, $\nu^{\epsilon}$ is an invariant measure of the solution map of Eq. \eqref{eq: truncated SFDEs epsilon} with the uniqueness established by \cite[Theorem 5.1]{Wu_Stochastic_2017}. Using the remote start method and the modified Yamada--Watanabe--Engelbert Theorem (Lemma \ref{lem: Yamada-Watanabe}), we obtain the strong stationary solution for the crude cocycle $(U^{\epsilon},\theta)$ as follows.
\begin{theorem} \label{thm: stationary solution}
 For any $\epsilon>0$, there is a pathwise unique stationary solution $\cY^{\epsilon}$ for crude cocycle $(U^{\epsilon},\theta)$ in $L^{\infty}(\R;L^2(\Omega,C_r))$. For any $t \in \R$ and $\tau \in (-\infty,0]$, $\cY^{\epsilon}(t)$ $\P$-a.s. satisfies
 \begin{equation} \label{eq: stationary solution}
  \begin{cases}
   \cY^{\epsilon}(t)(0) = \cY^{\epsilon}(t+\tau)(0)+\int_{t+\tau}^{t} b(\cY^{\epsilon}(s)) \, \d s+ \sqrt{\epsilon} \int_{t+\tau}^{t} \sigma(\cY^{\epsilon}(s)) \, \d W(s); \\
   \cY^{\epsilon}(t)(\tau) = \cY^{\epsilon}(t+\tau)(0),
  \end{cases}
 \end{equation}
 and has a uniform bound with time $t$ specifically $\sup_{t \in \R} \E \| \cY^{\epsilon}(t) \|_{r}^2 < +\infty$.
 Conversely, the uniform bound solution to \eqref{eq: stationary solution} is indeed the stationary solution for cocycle $(U^{\epsilon}, \theta)$. Furthermore, there exists a Borel measurable map $\mathcal{G}^{\epsilon}: \W \rightarrow \sC_r$ such that $\mathcal{G}^{\epsilon}(W)(\cdot)$ is an $\mathcal{F}_{t}$-adapted uniform bound solution to \eqref{eq: stationary solution}, called the strong stationary solution for the crude cocycle $(U^{\epsilon},\theta)$.
\end{theorem}
As far as we know, the stationary solution for SDEs and SFDEs was first introduced by It\^o and Nisio in \cite{Ito_stationary_1964} to explore the stochastic differential equations, including SFDEs. Using the notations in our paper, the stationary solution defined in \cite{Ito_stationary_1964} refers to the solution $\cY^{\epsilon}(t)(0)$ such that the finite-dimensional distribution of $(\cY^{\epsilon}(\cdot)(0), \d W)$ is invariant under the time shift. This implies that the stationary solution in \cite{Ito_stationary_1964} is a weak solution. Later, Bakhtin and Mattingly \cite{Bakhtin_Existence_2002, Bakhtin_Stationary_2005} extended these results. It is evident that the stationary solutions defined in Definition \ref{def: stationary solution} also qualify as the stationary solutions in \cite{Ito_stationary_1964}. Furthermore, the stationary solution of SFDE considered in the present paper is a strong solution and follows ``one force one solution'' from the perspective of random dynamic system. In particular, the strong stationary solution shown in Theorem \ref{thm: stationary solution} is essential for applying the weak convergence method of LDP.

Next, we introduce some basic concepts of LDP and uniform LDP. Let $\mathcal{E}_0$ and $\mathcal{E}$ be Polish spaces. Firstly, we define rate functions as follows, referring to \cite[Definition 1,4]{Budhiraja_Large_2008}.

\begin{definition}[Rate Function]
 A function $\cI: \mathcal{E} \rightarrow [0,+\infty]$ is called a rate function, if $\cI$ is lower semicontinuous. A rate function $\cI$ is called a good rate function if for each $c<+\infty$, the level set $\{\Phi \in \mathcal{E} : \cI(\Phi)\leq c\}$ is compact. For any Borel measurable set $A \in \mathcal{B}(\mathcal{E})$, we defined $\cI(A) :=\inf_{\Phi \in A} \cI(\Phi)$.
 A family of rate functions $\{\cI_{\xi}\}$ on $\sC_r$, parametrized by $\xi \in \mathcal{E}_0$, is said to have compact level sets on compacts if for all compact subsets $K \subset \mathcal{E}_0$ and each $M<+\infty$, $\Lambda_{M,K}= \cup_{\xi \in K} \{\Phi \in \mathcal{E}; \, \cI_{\xi}(\Phi) \leq M\}$ is a compact subset of $\mathcal{E}$.
\end{definition}
In this paper, we will set $\mathcal{E}_0 = C_r$ and $\mathcal{E} = \sC_r$ or $C_r$. We present the definition of LDP for a sequence of distributions on $\sC_r$; the standard definition can be found in \cite[Definition 2]{Budhiraja_Large_2008}.

\begin{definition}[Large Deviation Principle] \label{def: large deviation}
 We say a sequence $\{ X^{\epsilon}\}_{\epsilon>0} \subset \sC_r$ satisfies the large deviation principle with rate function $\cI$, if the following two conditions hold:
 \begin{itemize}
  \item [1.] Large deviation upper bound. For each closed subset $F$ of $\sC_r$,
   \begin{equation*}
    \limsup_{\epsilon \rightarrow 0} \epsilon\,\log \P(X^{\epsilon}\in F)\leq - \cI(F).
   \end{equation*}
  \item [2.] Large deviation lower bound. For each open subset $G$ of $\sC_r$,
   \begin{equation*}
    \liminf_{\epsilon \rightarrow 0} \epsilon\,\log \P(X^{\epsilon}\in G) \geq -\cI(G).
   \end{equation*}
 \end{itemize}
\end{definition}

Freidlin and Wentzell define the uniform LDP at the end of Section 3.3 of \cite{Freidlin_Random_2012}. This concept is applied to a sequence of distributions on $\sC_r$ with indexed space $C_r$ as follows.
\begin{definition}[Freidlin-Wentzell uniform LDP]
 Let $\mathcal{K}$ be the family of all compact subset of $C_r$. We say a sequence $\{ X^{\epsilon}_{\xi}\}_{\epsilon>0} \subset \sC_r$ indexed by $\xi \in C_r$ satisfies the uniform large deviation principle with rate function $\cI_{\xi}$, uniformly in $\mathcal{K}$ if the following two conditions hold:
 \begin{itemize}
  \item [1.] Upper bound. For any $K \in \mathcal{K}$, $\delta>0$ and $s_0>0$,
   \begin{equation*}
    \limsup_{\epsilon \rightarrow 0} \sup_{\xi \in K} \sup_{s \leq s_0} \Big\{ \epsilon \log \P \big( \inf_{ \{\Phi \in \sC_r; \, \cI_{\xi}(\Phi) \leq s\} } d_{r}(X^{\epsilon}_{\xi},\Phi) \geq \delta \big) + s \Big\} \leq 0.
   \end{equation*}
  \item [2.] Lower bound. For any $K \in \mathcal{K}$, $\delta>0$ and $s_0>0$,
   \begin{equation*}
    \liminf_{\epsilon \rightarrow 0} \inf_{\xi \in K} \inf_{ \{\Phi \in \sC_r; \, \cI_{\xi}(\Phi) \leq s_0\}} \Big\{ \epsilon\,\log \P \big( d_{r}(X^{\epsilon}_{\xi},\Phi) <\delta \big) + \cI_{\xi}(\Phi) \Big\} \geq 0.
   \end{equation*}
 \end{itemize}
\end{definition}

The LDP is well known to be equivalent to the Laplace principle, see \cite[Theorem 1.2.1, Theorem 1.2.3]{Dupuis_Weak_1997}. Salins et al. in \cite[Theorem 4.3]{Salins_Uniform_2019} show that the uniform LDP over compact sets is equivalent to the uniform Laplace principle. Given these equivalences, the rest of this work will investigate the Laplace principle and the uniform Laplace principle. They are stated as follows.
\begin{definition} [Laplace Principle] \label{def: Laplace principle}
We say the sequence $\{X^{\epsilon}\}_{\epsilon>0} \subset \sC_r$ satisfies the Laplace principle with rate function $\cI$, if for all bounded measurable functions $h: \sC_r \rightarrow \R$,
\begin{equation*}
 \lim_{\epsilon \rightarrow 0}\epsilon \log{\E}\Big[\exp\big\{-{\frac{1}{\epsilon}}h(X^{\epsilon})\big\} \Big]=-\inf_{ \Phi \in{ \sC_r }} \big\{h(\Phi)+\cI(\Phi)\big\}. 
\end{equation*}
\end{definition}
\begin{definition} [Uniform Laplace Principle] \label{def: uniform LP}
 Let $\{\cI_{\xi}\}$ be a family of rate functions on $\sC_r$ indexed by $\xi \in C_r$, and assume that this family has compact level sets on compacts. The family $\{ X^{\epsilon}_{\xi}\}_{\epsilon>0}$ is said to satisfy the uniform Laplace principle with rate function $\cI_{\xi}$, uniformly on compacts, if for all compact subsets $K \subset C_r$ and all bounded continuous functions $h:\sC_r \rightarrow \R$,
 \begin{equation*}
  \lim_{\epsilon \rightarrow 0} \sup_{\xi \in K} \Big| \epsilon \log{\E}\Big[\exp\big\{-{\frac{1}{\epsilon}}h(X^{\epsilon}_{\xi})\big\} \Big] + \inf_{ \Phi \in{ \sC_r }} \big\{h(\Phi)+\cI_{\xi}(\Phi)\big\} \Big|=0. 
 \end{equation*}
\end{definition}

To give the rate functions, we need to consider the skeleton equation as follows
\begin{equation} \label{eq: skeleton equation}
 Y^{0,v}_{t; t_0, \xi}(0)= \begin{cases}
 \xi(0)+\int_{t_0}^{t} b(Y^{0,v}_{s;t_0,\xi}) \, \d s + \int_{t_0}^{t} \sigma(Y^{0,v}_{s;t_0,\xi}) \cdot v(s) \, \d s, \quad & \mbox{if} \quad t \geq t_0; \\
 \xi(t-t_0), \quad & \mbox{if} \quad t < t_0,
 \end{cases}
\end{equation}
and $Y^{0,v}_{t;t_0,\xi}(\tau) =Y^{0,v}_{t+\tau;t_0,\xi}(0)$ for any $\tau \leq 0$, where $t_0 \in \R$, $\xi \in C_r$ and $v \in L^2(\R;\R^m)$. Lemma \ref{lem: well-posedness-skeleton-equation} gives the well-posedness of Eq. \eqref{eq: skeleton equation} and shows that solution map $\{Y^{0,v}_{\cdot; -n, \xi}\}_{n \in \N}$ converge to $\mathcal{Y}^{0,v,\ast}\in \sC_r$ as $n \rightarrow +\infty$. The limit $\mathcal{Y}^{0,v,\ast}$ satisfies that for any $\tau \leq 0$ and $t_2>t_1$, 
\begin{equation} \label{eq: skeleton equation 0}
 \mathcal{Y}^{0,v,\ast}(t_2)(\tau) = \mathcal{Y}^{0,v,\ast}(t_1)(\tau) + \int_{t_1+\tau}^{t_2+\tau} b(\mathcal{Y}^{0,v,\ast}(s) ) \, \d s + \int_{t_1+\tau}^{t_2+\tau} \sigma(\mathcal{Y}^{0,v,\ast}(s)) \cdot v(s) \, \d s.
\end{equation} 
For each $ \frac{\d}{\d t} V \in L^{2}(\R;\R^m)$ and $v=\frac{\d}{\d t} V$, define $\mathcal{G}^{0}_{t_0,\xi} (V)$ and $\mathcal{G}^{0}_{-\infty} (V)$ as
\begin{equation*}
 \mathcal{G}^{0}_{t_0,\xi} (V)(\cdot):= Y^{0,v}_{\cdot;t_0,\xi}; \quad \mathcal{G}^{0}_{-\infty} (V)(\cdot):= \mathcal{Y}^{0,v,\ast}(\cdot).
\end{equation*}
For each $\Phi \in \sC_r$, define rate functions $\cI_{t_0,\xi}(\Phi)$ and $\cI_{-\infty}(\Phi)$ as
\begin{align*}
 \cI_{t_0,\xi} (\Phi) &:= \inf_{\big\{ v \in L^{2}(\R ; \R^m); \; \Phi ={\mathcal G}^{0}_{t_0,\xi} (\int_{0}^{\cdot} v(s) \d s) \big\}} \Big\{\frac{1}{2} \int_{-\infty}^{+\infty} |v(s)|^{2} \d s \Big\}; \\
 \cI_{-\infty}(\Phi) & := \inf_{\big\{ v \in L^{2}(\R ; \R^m); \; \Phi ={\mathcal G}^{0}_{-\infty} (\int_{0}^{\cdot} v(s) \d s) \big\}} \Big\{\frac{1}{2} \int_{-\infty}^{+\infty} |v(s)|^{2} \d s \Big\},
\end{align*}
where the infimum over the empty set is taken to be $+\infty$. In Lemma \ref{lem: proposition-skeleton-equation-1}, we will show that $\cI_{t_0,\xi}$ and $\cI_{-\infty}$ are good rate functions. In Lemma \ref{lem: proposition-skeleton-equation-2}, we will further prove that $\cI_{t_0,\xi}$ indexed by $\xi \in C_r$ have compact level sets on compacts.

With the above notion, we state our main results: the uniform LDP for solution maps $\{ Y_{\cdot;t_0,\xi}^{\epsilon}\}_{\epsilon>0}$ of Eq. \eqref{eq: truncated SFDEs epsilon} and the LDP for stationary solutions $ \{ \mathcal{Y}^{\epsilon} \}_{\epsilon>0}$.
\begin{theorem} \label{thm: uniform-LDP}
 For each $t_0 \in \R$, solution maps $\{ Y_{\cdot;t_0,\xi}^{\epsilon}\}_{\epsilon>0}$ of Eq. \eqref{eq: truncated SFDEs epsilon} satisfies the uniform Laplace principle with rate functions $\{ \cI_{t_0,\xi}: \, \xi \in C_r \}$ that have compact level sets on compacts.
\end{theorem}
\begin{theorem} \label{thm: LDP-stationary solution}
 The stationary solution maps $\{ \cY^{\epsilon}\}_{\epsilon>0}$ for the crude cocycle $(U^{\epsilon}, \theta)$ satisfy the Laplace principle with the good rate function $\cI_{-\infty}$.
\end{theorem}

Let $\Pi: \sC_{r} \rightarrow C_r$ be the projection operator that maps $\Phi \in \sC_r$ to $\Phi(0) \in C_r$. Define the rate function $\cI^{\prime}(\phi) := \inf_{\Phi \in \sC_r} \{ \cI_{-\infty}(\Phi): \,\Pi \Phi= \phi \}$ for any $\phi \in C_r$. If $m=d$ and $\sigma(\phi)$ is an invertible matrix in $\R^{d\times d}$ for any $\phi \in C_r$, this rate function takes the following form:
\begin{equation*}
  \cI^{\prime}(\phi)= \frac{1}{2} \int_{-\infty}^{0} \Big| \big(\dot{\phi}(s)-b(\phi(\cdot+s)) \big)^{T} \big( \sigma \sigma^{T}(\phi(\cdot+s)) \big)^{-1} \big(\dot{\phi}(s)-b(\phi(\cdot+s)) \big) \Big| \, \d s,
\end{equation*}
which, in its structure, does not involve taking the infimum over different orbits, in contrast to one derived from the quasi-potential.
Using the contraction principle (cf. \cite[Theorem 4.2.1]{Dembo_Large_2010}), we obtain the LDP for invariant measure $ \nu^{\epsilon} := \P \circ (\mathcal{Y}^{\epsilon}(0))^{-1}$ rate function with $\cI^{\prime}$.  
\begin{corollary} \label{cor: LDP invariance measure}
 $\cI^{\prime}: C_r \mapsto [0,+\infty]$ is a good rate function. The family of invariant measures $\{ \nu^{\epsilon}\}_{\epsilon>0}$ satisfies the large deviation principle with $\cI^{\prime}$, meaning the following holds:
 \begin{itemize}
  \item [1.] Large deviation upper bound. For each closed subset $F$ of $C_r$,
   \begin{equation*}
    \limsup_{\epsilon \rightarrow 0} \epsilon\,\log \nu^{\epsilon}(F) \leq - \cI^{\prime}(F).
   \end{equation*}
  \item [2.] Large deviation lower bound. For each open subset $G$ of $C_r$,
   \begin{equation*}
    \liminf_{\epsilon \rightarrow 0} \epsilon\,\log \nu^{\epsilon}(G) \geq -\cI^{\prime}(G).
   \end{equation*}
 \end{itemize}
\end{corollary}
By the same way, the contraction principle deduces the LDP for solutions $\{ Y^{\epsilon}(\cdot;t_0,\xi)\}_{\epsilon>0}$ of Eq. \eqref{eq: truncated SFDEs epsilon} from the uniform LDP for solution maps $\{ Y_{\cdot;t_0,\xi}^{\epsilon}\}_{\epsilon>0}$ shown in Theorem \ref{thm: uniform-LDP}. As discussed after Theorem \ref{thm: stationary solution}, $\cY^{\epsilon}(\cdot)(0)$ is the stationary solution defined in \cite{Ito_stationary_1964} and the marginal distribution of $\cY^{\epsilon}(t)(0)$ is invariant under the time shift. Using the contraction principle again, the LDP for the marginal distributions of $\{\cY^{\epsilon}(t)(0)\}_{\epsilon}$ comes from Corollary \ref{cor: LDP invariance measure}.

\section{SFDE and stationary solution} \label{sec: SFDE and Stationary solution}

In the section, we will show the existence and pathwise uniqueness of solution maps to Eq. \eqref{eq: truncated SFDEs epsilon} in Section \ref{subsec: SFDE} and construct the stationary solution in Section \ref{subsec: stationary solution}.

\subsection{SFDE} \label{subsec: SFDE}
In this subsection, we outline the existing results on the well-posedness of Eq. \eqref{eq: truncated SFDEs epsilon} and proved the existence of Borel measurable map $\mathcal{G}^{\epsilon}_{t_0}: \W \times C_r \rightarrow \sC_r$ such that $\mathcal{G}^{\epsilon}_{t_0}(W,\xi) \in \sC_{r}$ is $\cF_{t}$-adapted and is the solution map of Eq. \eqref{eq: truncated SFDEs epsilon} with initial data $\xi \in C_r$.
\begin{definition} \label{def: solution}
For fixed $t_0 \in \R$, a continuous $\R^d$-valued and ${\mathcal{F}}_{t}$-adapted process $Y^{\epsilon} (t; t_0, \xi) $ is called a global solution to Eq. \eqref{eq: truncated SFDEs epsilon} with initial data $\xi \in C_r$, if 
\begin{equation*}
 \begin{cases}
  Y^{\epsilon} (t; t_0, \xi) = \xi(0)+\int_{t_0}^{t} b(Y^{\epsilon}_{s;t_0,\xi}) \, \d s+\int_{t_0}^{t} \sigma(Y^{\epsilon}_{s;t_0,\xi}) \, \d W(s) \quad \mbox{a.s.}, & \quad \forall t \geq t_0; \\
  Y^{\epsilon}(t; t_0, \xi)= \xi(t-t_0), & \quad \forall t \leq t_0,
 \end{cases}
\end{equation*}
and the segment process $Y_{t;t_0,\xi}^{\epsilon} := \left\{ Y^{\epsilon}(t+\tau;t_0,\xi); \, \tau \in (-\infty,0] \right\} \in C_r$. And we call $Y_{t;t_0,\xi}^{\epsilon}$ as the global solution map of Eq. \eqref{eq: truncated SFDEs epsilon} with initial data $\xi \in C_r$.
\end{definition}

Recall the assumptions (H1-H3) of continuous functionals $b:C_{r} \rightarrow \R^{d}$ and $\sigma:C_{r} \rightarrow \R^{d\times m}$ in Section \ref{sec: Preliminary}. Under assumptions (H1-H3), Wu et al. \cite[Theorem 3.2]{Wu_Stochastic_2017} give the existence of the global solution $Y^{\epsilon}(t;t_0,\xi)$ to Eq. \eqref{eq: truncated SFDEs epsilon}. Furthermore, with more rigorous coefficient conditions, these solutions remain uniformly bounded in $L^2(\Omega, C_r)$ and are unique. We summarize the results of \cite[Section 3]{Wu_Stochastic_2017} as the following Lemma \ref{lem: truncated-solution}, which emphasizes the dependence of each estimated coefficient on $\epsilon$.
\begin{lemma} \label{lem: truncated-solution}
Under assumptions (H1-H3), for any $t_0 \in \R$, initial data $\xi\in C_{r}$ and $\epsilon \in [0,1]$, 
\begin{enumerate}[i)]
 \item Eq. \eqref{eq: truncated SFDEs epsilon} has a global solution $Y^{\epsilon}(t; t_0, \xi)$ on $[t_0,+\infty)$;
 
 \item furthermore, if the coefficients $\lambda_{1},\lambda_{2}$ and $\lambda_3$ satisfy $2\lambda_{1}>2\lambda_{2}\mu_{1}^{(2r)}+ \epsilon \lambda_{3}\mu_{2}^{(2r)}$, then there exists $\lambda \in (0,(2\lambda_{1}-2\lambda_{2}\mu_{1}^{(2)}- \epsilon \lambda_{3}\mu_{2}^{(2r)} ) \wedge 2r )$ such that 
 \begin{equation*}
  \E |Y^{\epsilon}(t; t_0, \xi)|^{2}\leq C_{1}(\epsilon, \lambda, \xi) \, e^{-\lambda (t-t_0)}+C_{2}(\epsilon,\lambda, \xi), \quad \forall \, t \geq t_0,
 \end{equation*}
 where $C_{1}(\epsilon, \lambda, \xi)$ and $C_{2}(\epsilon, \lambda, \xi)$ are uniformly bounded about coefficient $\epsilon$; 
 
 \item under the conditions in ii), the solutions $Y^{\epsilon}(t; t_0, \xi_1)$ and $Y^{\epsilon}(t; t_0, \xi_2)$ to \eqref{eq: truncated SFDEs epsilon} with different initial data $\xi_{1} \in C_r$ and $\xi_{2} \in C_r$ satisfy
 \begin{equation*}
  \E \big| Y^{\epsilon}(t; t_0, \xi_1)-Y^{\epsilon}(t; t_0, \xi_2) \big|^{2}\leq C_{3}(\epsilon, \lambda) \, \big\|\xi_{1}-\xi_{2} \big\|_{r}^{2} \, e^{-\lambda (t-t_0)}, \quad \forall \, t \geq t_0,
 \end{equation*}
 where $C_{3}(\epsilon, \lambda)$ is uniformly bounded. Thus, SFDEs \eqref{eq: truncated SFDEs epsilon} admit unique global solutions. 
\end{enumerate}
\end{lemma}
We always assume (H1-H3) and the conditions in Lemma \ref{lem: truncated-solution} ii) hold in what follows to ensure that Eqs. \eqref{eq: truncated SFDEs epsilon} admit a uniformly bounded global solution. More details about $C_{1}(\epsilon, \lambda, \xi)$, $C_{2}(\epsilon, \lambda, \xi)$ and $C_{3}(\epsilon, \lambda)$ can be found in \cite[Theorem 3.2]{Wu_Stochastic_2017}. To simplify, we always write these constants as $C_{1,\epsilon}$, $C_{2,\epsilon}$ and $C_{3,\epsilon}$ respectively. In \cite[Section 4]{Wu_Stochastic_2017}, Wu et al. give some properties of solution map $Y^{\epsilon}_{t;t_0,\xi}$, including the continuity, strong Markov property, mean-square boundedness and so on. We summarize the mean-square boundedness of solution map $Y^{\epsilon}_{t;t_0,\xi}$ as the following Lemma \ref{lem: truncated-solution-map} and highlight that each estimated coefficient depends on $\epsilon$.
 
\begin{lemma} \label{lem: truncated-solution-map}
 Under the above assumptions, for any fixed $\epsilon>0$, $t_0 \in \R$ and initial data $\xi \in C_r$, solution map $Y^{\epsilon}_{t;t_0,\xi}$ is a continuous, $\mathcal{F}_t$-adapted and strong homogeneous Markov process. For $\lambda \in(0,(2\lambda_{1}-2\lambda_{2}\mu_{1}^{(2r)}- \epsilon \lambda_{3}\mu_{2}^{(2r)}) \wedge 2r)$, there exist $C_{4}(\epsilon,\lambda,\xi),C_{5}(\epsilon,\lambda,\xi) \geq 0$ such that
 \begin{equation*}
  \mathbb{E}\|Y^{\epsilon}_{t;t_0,\xi}\|_{r}^{2}\leq C_{4}(\epsilon,\lambda,\xi) \, e^{-\lambda (t-t_0)}+C_{5}(\epsilon,\lambda,\xi), \quad \forall \, t \geq t_0.
 \end{equation*}
For some $0<\epsilon_0< \lambda$, let $\lambda_{\epsilon_0}= \lambda- \epsilon_0$. Then there exist a constant $C_6(\epsilon, \lambda, \epsilon_0)>0$ such that solution maps $Y^{\epsilon}_{t;t_0,\xi_1}$ and $Y^{\epsilon}_{t;t_0,\xi_2}$ with the different initial data $\xi_1$ and $\xi_{2}\in C_r$ satisfy 
\begin{equation} \label{eq: cauchy-truncated SFDEs solution map}
 \mathbb{E} \big\|Y^{\epsilon}_{t;t_0,\xi_1}-Y^{\epsilon}_{t;t_0,\xi_2} \big\|_{r}^{2}\leq C_6(\epsilon, \lambda, \epsilon_0) \, \big\|\xi_{1}-\xi_{2} \big\|_{r}^{2} \, e^{-\lambda_{\epsilon_0} (t-t_0)}, \quad \forall \, t \geq t_0.
\end{equation}
These constants $C_{4}(\epsilon, \lambda, \xi)$, $C_{5}(\epsilon, \lambda, \xi)$ and $C_{6}(\epsilon, \lambda, \epsilon_0)$ are uniformly bounded with respect to coefficient $\epsilon$, and we briefly note them as $C_{4,\epsilon}$, $C_{5,\epsilon}$ and $C_{6,\epsilon}$ respectively.
\end{lemma}

The pathwise uniqueness shown in Lemma \ref{lem: truncated-solution-map} implies the existence of the strong solution map $Y^{\epsilon}_{t;t_0,\xi}$ in probability sense. To prove it, we need Lemma \ref{lem: Yamada-Watanabe}, a modified version of the general Yamada–Watanabe–Engelbert Theorem (see \cite[Theorem 3.14]{Kurtz_YamadaWatanabeEngelbert_2007}), which applies to $\cF_{t}$-adapted solution rather than compatible one.
\begin{lemma} \label{lem: truncated-measurable-map}
 For any fixed $\epsilon>0$ and $t_0 \in \R$ and for any $\mathcal{F}_{t_0}$-measurable random initial data $\xi \in C_r$, Eq. \eqref{eq: truncated SFDEs epsilon} has a unique $\mathcal{F}_{t}$-adapted strong solution map. It means that there is a Borel measurable map $\mathcal{G}^{\epsilon}_{t_0}: \W \times C_r \rightarrow \sC_r$ such that for any $\xi \in C_r$, $\mathcal{G}^{\epsilon}_{t_0}(W,\xi) \in \sC_{r}$ is $\cF_{t}$-adapted for any $t \geq t_0$ and $\P$-a.s. satisfies Eq. \eqref{eq: truncated SFDEs epsilon} with initial data $\xi \in C_r$.
\end{lemma}
\begin{proof}
 We will use the pathwise uniqueness of $\mathcal{F}_{t}$-adapted solution map $Y^{\epsilon}_{\cdot;t_0,\xi}$ and the Yamada--Watanabe--Engelbert Theorem (Lemma \ref{lem: Yamada-Watanabe}) to prove the existence of strong solutions. Let $(Y^{\epsilon}_{\cdot;t_0,\xi},W)$ and $(\Tilde{Y}^{\epsilon}_{\cdot;t_0,\Tilde{\xi}},\Tilde{W})$ be two solution maps of Eq. \eqref{eq: truncated SFDEs epsilon} on the same stochastic basis $(\Omega,{\mathcal{F}},\mathbb{P},({\mathcal{F}}_{t}))$ and with the double-side Wiener process $W(t)$ on $(\Omega,{\mathcal{F}},\mathbb{P})$. If the random initial data $\xi=\Tilde{\xi}$ $\P$-a.s., using the estimate \eqref{eq: cauchy-truncated SFDEs solution map} and the strong Markov property, we get $\P$-a.s.
 \begin{equation*}
  Y^{\epsilon}_{m;t_0,\xi}=\Tilde{Y}^{\epsilon}_{m;t_0,\Tilde{\xi}}, \quad \forall \, m \in \N \quad \mbox{and} \quad m \geq t_0.
 \end{equation*}
 By the definition of $Y^{\epsilon}_{t;t_0,\xi}$ and Proposition \ref{prop: cr} ii), we obtain that $\P$-a.s. $Y^{\epsilon}_{t;t_0,\xi}=\Tilde{Y}^{\epsilon}_{t;t_0,\Tilde{\xi}}$ for any $t >t_0$, which means the pathwise uniqueness.
 
Similar to \cite[Example 3.9]{Kurtz_YamadaWatanabeEngelbert_2007}, Eq. \eqref{eq: truncated SFDEs epsilon} gives a constraint $\Gamma_{t_0}^{\epsilon}$ that determines a convex subset $S_{\Gamma_{t_0}^{\epsilon}}$ of the Borel probability measure space $\mathcal{P}( \W \times C_r \times \sC_r)$. We say that $\varrho \in S_{\Gamma_{t_0}^{\epsilon}}$ with constraint $\Gamma_{t_0}^{\epsilon}$, if there exists random elements $(w, \tilde{\xi},X) \in \W \times C_r \times \sC_r$ such that $\varrho= \P \circ (w,\tilde{\xi},X)^{-1}$ with marginal distribution $\P \circ (W,\xi)^{-1} $ on $\W \times C_r$ and $(\tilde{\xi},w,X)$ satisfies
 \begin{align*}
  \lim_{n \rightarrow + \infty} \E \Bigg[ & \bigg| X(t)(0) -\tilde{\xi}(0) - \int_{t_0}^{t} b \big(X(s) \big) \, \d s \bigg. \Bigg. \\
  \Bigg. \bigg. & -\sqrt{\epsilon} \sum_{k=0}^{\lfloor n (t-t_0) \rfloor} \sigma \Big( X \Big(t_0+ \frac{k}{n} \Big) \Big) \cdot \Big( w \Big( t \wedge \big(t_0+\frac{k+1}{n} \big) \Big)-w \Big(t_0+ \frac{k}{n} \Big) \Big) \bigg| \Bigg] =0,
 \end{align*}
 if $t>t_0$; $X(t)(0)=\tilde{\xi}(t-t_0)$ if $t \leq t_0$; and $X(t)(\tau)=X(t+\tau)(0)$ for all $\tau \in (-\infty,0]$. Thus, $S_{\Gamma_{t_0}^{\epsilon}}$ is convex. Applying the pathwise uniqueness and Lemma \ref{lem: Yamada-Watanabe}, we obtain that Eq. \eqref{eq: truncated SFDEs epsilon} has a unique $\mathcal{F}_{t}$-adapted strong solution map.
\end{proof}

\subsection{Stationary solution} \label{subsec: stationary solution}

In this subsection, we will prove Theorem \ref{thm: stationary solution}, which shows the well-posedness of the stationary solution $\mathcal{Y}^{\epsilon}$ for $(U^{\epsilon},\theta)$. Refer to Definition \ref{def: stationary solution} for the definition of a stationary solution.

\begin{proof}[\bf Proof of Theorem \ref{thm: stationary solution}]
 We will divide the proof into the following several steps.

 {\bf Step 1: construct stationary solution.}
 Using estimate \eqref{eq: cauchy-truncated SFDEs solution map} and the strong Markov property of $Y^{\epsilon}_{t;t_0,\xi}$, for any $m>n>-t$, we have
 \begin{align}
  \E \big\| Y^{\epsilon}_{t;-m,\xi}- Y^{\epsilon}_{t;-n,\xi} \big\|_{r}^{2} & =\E \big\| Y^{\epsilon}_{t;-n,Y^{\epsilon}_{-n;-m,\xi}}-Y^{\epsilon}_{t;-n,\xi} \big\|_{r}^{2} \nonumber \\ 
  & \leq C_{6,\epsilon} \E \big\| Y^{\epsilon}_{-n;-m,\xi}-\xi \big\|_{r}^{2} e^{-\lambda_{\epsilon_0} (n+t)} \nonumber \\ 
  & \leq 2 C_{6,\epsilon}(C_{4,\epsilon}+C_{5,\epsilon}) e^{-\lambda_{\epsilon_0} (n+t)}. \label{eq: cauchy-sequence}
 \end{align}
 Thus, for any fixed $t \in \R$, $\{ Y^{\epsilon}_{t;-n,\xi} \}_{n \in \N}$ is a Cauchy sequence in $L^2(\Omega,C_r)$ with limit $Y_t^{\epsilon,\ast}$. Let $\cY^{\epsilon}(t, \omega)(\cdot) :=Y^{\epsilon,\ast}_{t}(\cdot,\omega)$, then by Proposition \ref{prop: cr} i), we obtain $\P$-a.s.
 \begin{align*}
  \cY^{\epsilon}(t,\omega)(\tau) = \lim_{n \rightarrow +\infty} Y^{\epsilon}_{t;-n,\xi}(\tau) = \lim_{n \rightarrow +\infty} Y^{\epsilon}_{t+\tau;-n,\xi}(0) = \cY^{\epsilon}(t+\tau,\omega)(0),
 \end{align*}
 for any fixed $t \in \R$ and $\tau\leq 0$. The point-wise convergence also shows that $\cY^{\epsilon}(t,\omega)(\cdot+\tau)=\cY^{\epsilon}(t+\tau,\omega)(\cdot) $. By Proposition ii), we know that $\cY^{\epsilon}(t)$ is a continuous process and $\cY^{\epsilon} \in \sC_r$. 
 Estimate \eqref{eq: cauchy-truncated SFDEs solution map} shows $\{ Y^{\epsilon}_{\cdot;-n,\xi} \}_{n}$ converge to $\cY^{\epsilon}$ in $L^2(\Omega, \sC_r)$ as $n \rightarrow +\infty$. By Lemma \ref{lem: truncated-solution-map}, $\{ Y^{\epsilon}_{t;-n,\xi} \}_{n \geq |t|}$ are $\mathcal{F}_t$ adapted, so the limit $\cY^{\epsilon}(t)=Y_t^{\epsilon,\ast}$ is $\mathcal{F}_t$ adapted. Combining the adaptivity and continuity of $\cY^{\epsilon}(t)$, we obtain that $\cY^{\epsilon}(t)$ is an $\mathcal{F}_t$ progressively measurable process.

 By the definition of the dynamical system $(U^{\epsilon},\theta)$ and the strong Markov property, for any fixed $s>-n$, we obtain that $\P$-a.s. for any $t \geq 0$, 
 \begin{align*}
  U_0^{\epsilon}(t, Y^{\epsilon}_{s;-n,\xi}(\omega),\theta(s,\omega)) & = \big\{ Y^{\epsilon}(t+\tau; 0, Y^{\epsilon}_{s;-n,\xi}(\omega))(\theta(s,\omega)) ; \, \tau \in (-\infty,0] \big\} \\
  & = \big\{ Y^{\epsilon}(t+s+\tau; -n,\xi))(\omega) ; \, \tau \in (-\infty,0] \big\}= Y^{\epsilon}_{t+s;-n,\xi}(\omega).
 \end{align*}
 By estimate \eqref{eq: cauchy-truncated SFDEs solution map}, we obtain that $U_0^{\epsilon}(t, \xi,\omega)$ is continuous about initial data $\xi \in C_r$. For fixed $s \in\R$, using this continuous dependence and the strong Markov property, we get that $\P$-a.s.
 \begin{align*}
  U_0^{\epsilon}(t, Y_s^{\epsilon,\ast}(\omega),\theta(s,\omega)) &= U_0^{\epsilon}(t, \lim_{n \rightarrow +\infty} Y^{\epsilon}_{s;-n,\xi}(\omega),\theta(s,\omega)) = \lim_{n \rightarrow +\infty} U_0^{\epsilon}(t, Y^{\epsilon}_{s;-n,\xi}(\omega),\theta(s,\omega)) \\
  & = \lim_{n \rightarrow +\infty} Y^{\epsilon}_{t+s;-n,\xi}(\omega) = Y_{t+s}^{\epsilon,\ast}(\omega), \quad \forall \, t \in \R .
 \end{align*}
 Due to $\cY^{\epsilon}(t, \omega)(\cdot)=Y^{\epsilon,\ast}_{t}(\cdot,\omega)$, we get that $U_0^{\epsilon}(t, \mathcal{Y}^{\epsilon}(s,\omega), \theta(s, \omega))=\mathcal{Y}^{\epsilon}(t+s,\omega)$ $\P$-a.s. holds. 
 On the other hand, by the definition of the Wiener shift $\theta$, we have $Y^{\epsilon}(\cdot; -n, \xi)( \theta(s,\omega))=Y^{\epsilon}(\cdot+s; -n+s, \xi)( \omega)$. So for the solution map, we obtain the following identity
 \begin{equation*}
  Y^{\epsilon}_{t; -n, \xi}(\tau, \theta(s,\omega))= Y^{\epsilon}_{t+s; -n+s, \xi}(\tau, \omega), \quad \forall \, \tau \in (-\infty,0].
 \end{equation*}
 Taking $n \rightarrow +\infty$, then $\mathcal{Y}^{\epsilon}(t+s,\omega) = Y_{t+s}^{\epsilon,\ast}(\omega)=Y_{t}^{\epsilon,\ast}(\theta(s,\omega))=\mathcal{Y}^{\epsilon}(t,\theta(s,\omega))$ $\P$-a.s. holds. Combining the above results, we proved that $\cY^{\epsilon}$ is the stationary solution for $(U^{\epsilon}, \theta)$. Furthermore, by the definition of $U_{t_0}^{\epsilon}$ and equality \eqref{ident: trans}, we get $\cY^{\epsilon}(t)$ satisfying Eq. \eqref{eq: stationary solution}. Taking $t_0 \rightarrow - \infty$, Lemma \ref{lem: truncated-solution-map} yields a uniform bound, specifically $\E \| \cY^{\epsilon}(t) \|_{r}^2 \leq C_5(\epsilon,\lambda,\xi) $ for all $t \in \R$.

 {\bf Step 2: the pathwise uniqueness.}
 If $\Tilde{\cY}^{\epsilon}$ is another stationary solution for the crude cocycle $(U^{\epsilon},\theta)$, then for any fixed $t \in \R$, we obtain 
 \begin{align*}
  \E \| \Tilde{\cY}^{\epsilon}(t,\omega)-\cY^{\epsilon}(t,\omega) \|_{r}^2 \! & = \! \E \| U_0^{\epsilon} (t+n, \Tilde{\mathcal{Y}}^{\epsilon}(-n,\omega), \theta(-n, \omega)) \! - \! U_0^{\epsilon} (t+n, \mathcal{Y}^{\epsilon}(-n,\omega), \theta(-n, \omega)) \|_{r}^2 \\
  & = \! \E \big\| Y^{\epsilon}_{t;-n,\Tilde{\cY}^{\epsilon}(-n,\omega)}-Y^{\epsilon}_{t;-n,{\cY}^{\epsilon}(-n,\omega)} \big\|_{r}^2, \quad \forall n \in \Z_{+} \cap \big[ |t|,+\infty \big).
 \end{align*}
 Analogous to estimate \eqref{eq: cauchy-sequence}, $\E \| \cY^{\epsilon}(-n,\omega) \|_{r}^2 $ is bounded and independent of $n$. Thus,
 \begin{equation*}
   \E \| \Tilde{\cY}^{\epsilon}(t,\omega)-\cY^{\epsilon}(t,\omega) \|_{r}^2 \leq 2 C_{6,\epsilon}(C_{4,\epsilon}+C_{5,\epsilon}) e^{-\lambda_{\epsilon_0} (n+t)}, \quad \forall n \in \Z_{+} \cap [|t|,+\infty).
 \end{equation*}
 Taking $n \rightarrow +\infty$, we obtain that for any $t \in \R$, $\Tilde{\cY}^{\epsilon}(t,\omega)=\cY^{\epsilon}(t,\omega)$, $\P$-a.s. holds. Applying Proposition \ref{prop: cr} ii), we get the pathwise uniqueness of the stationary solution. Note that to prove the pathwise uniqueness, we only used conditions: $\sup_{t \in \R} \E\|X(t)\|_{r}^2<+\infty$ and 
 \begin{equation*} 
   U_0^{\epsilon}(t, X(s,\omega), \theta(s, \omega))=X(t+s,\omega), \quad \forall \, t \geq 0, \quad \forall \, s \in \R, \quad \P\mbox{-a.s.}. 
 \end{equation*}
  So the uniform bound solution $X$ to Eq. \eqref{eq: stationary solution}, which naturally satisfies the above two conditions, is indeed the stationary solution for $(U^{\epsilon}, \theta)$, based on uniqueness.

 {\bf Step 3: strong stationary solution.} Eq. \eqref{eq: stationary solution} gives a constraint $\Gamma^{\epsilon}$ that determine $S_{\Gamma^{\epsilon}} \subset \mathcal{P}( \W \times \sC_r)$. We say that $\varrho \in S_{\Gamma^{\epsilon}}$ with constraint $\Gamma^{\epsilon}$, if exists $(w, X) \in \W \times \sC_r$ such that $\varrho= \P \circ (w,X)^{-1}$ with marginal distribution $\P \circ (W)^{-1} $ on $\W $ and $(X,w)$ satisfies
 \begin{align*}
  \lim_{n \rightarrow + \infty} \E \bigg[ & \bigg| X(t)(0) -X(t_0)(0) - \int_{t_0}^{t} b \big(X(s) \big) \, \d s \bigg. \bigg. \\
  \bigg. \bigg. & -\sqrt{\epsilon} \sum_{k=0}^{\lfloor n (t-t_0) \rfloor} \sigma \Big( X \Big(t_0+ \frac{k}{n} \Big) \Big) \cdot \Big( w \Big( t \wedge \big(t_0+\frac{k+1}{n} \big) \Big)-w \Big(t_0+ \frac{k}{n} \Big) \Big) \bigg| \bigg] =0,
 \end{align*}
 for any $-\infty<t_0<t<+\infty$; $X(t)(\tau)=X(t+\tau)(0)$ for all $t \in \R$ and $\tau \in (-\infty,0]$; and $\E\|X(t)\|_{r}^2=\E\|X(s)\|_{r}^2<+\infty $ for any $t,s \in \R$. It is clear that $S_{\Gamma}$ is convex. By the pathwise uniqueness and the Yamada--Watanabe Theorem (Lemma \ref{lem: Yamada-Watanabe}), there is a Borel measurable $\mathcal{G}^{\epsilon}: \W \rightarrow \sC_r$ such that $\mathcal{G}^{\epsilon}(W)$ is $\mathcal{F}_{t}$-adapted, uniform bounded and satisfies Eq. \eqref{eq: stationary solution}. Thus, $\mathcal{G}^{\epsilon}(W)$ is indeed the stationary solution for $(U^{\epsilon}, \theta)$.
\end{proof}

\section{Large Deviations Principle} \label{sec: Large Deviation}

In this section, we will use the weak converge approach to prove the uniform LDP for $\{ Y^{\epsilon}_{\cdot;t_0,\xi}\}_{\epsilon>0}$ to SFDEs \eqref{eq: truncated SFDEs epsilon} and the LDP for the stationary solution $\{\mathcal{Y}^{\epsilon}\}_{\epsilon>0}$ for $(U^{\epsilon},\theta)$. In Section \ref{subsec: skeleton equation}, we will show the well-posedness of skeleton equation \eqref{eq: skeleton equation} and rate functions $\cI_{t_0,\xi}$ have compact level sets on compacts. In Section \ref{subsec: uniform LDP}, we will show the well-posedness and compactness of controlled SFDEs, culminating in the proof of Theorem \ref{thm: uniform-LDP}. In Section \ref{subsec: LDP for stationary solution}, through the convergence of controlled SFDEs, we will prove Theorem \ref{thm: LDP-stationary solution}.

The weak converge approach is based on the variational representation and the existence of strong solutions. The variational representation for a certain functional of double-side Wiener process $W$ is similar to the functional of the Brownian motion, see details in Lemma \ref{lem: variational representation}. The existence of strong solution has been shown in Lemma \ref{lem: truncated-measurable-map} and Theorem \ref{thm: stationary solution}. Thus, following \cite[Theorem 4.4]{Budhiraja_variational_2000} and \cite[Theorem 5]{Budhiraja_Large_2008}, there are sufficient conditions of the (uniform) Laplace principle. We will list and verify these conditions in Section \ref{subsec: uniform LDP} and \ref{subsec: LDP for stationary solution}.

Before that, we need to introduce some notions used in this section. Let $\mathcal{A}$ denote the Cameron–Martin space on $\R$, that is,
\begin{equation} \label{def: variational space}
 \cA:= \big\{ v: \; v \; \mbox{is} \; \R^m\mbox{-valued predictable processes and} \int_{-\infty}^{+\infty} |v(s)|^2 \d s < + \infty,\, \P\mbox{-}a.s. \big\}.
\end{equation}
For any $v \in \cA$, we set the map $\int_{0}^{\cdot} v(s) \d s: \R \rightarrow \R^{m}$ as 
\begin{equation*}
 t \in \R \mapsto \begin{cases}
  \int_{0}^{t} v(s) \d s, & \quad \mbox{when} \quad t \geq 0; \\
  - \int_{t}^{0} v(s) \d s, & \quad \mbox{when} \quad t < 0.
 \end{cases}
\end{equation*}
For fixed $M>0$, we define $S_M:= \{v \in L^2(\R; \R^m); \, \int_{-\infty}^{+\infty} |v(s)|^2 \d s < M \}$ and $\cA_{M} := \{ v \in \cA ; \, v(\cdot)(\omega) \in S_M \} $. Observe that $S_{M}$ is a compact Polish space under the weak topology in $L^2(\R; \R^{m})$ and $S_{M}$ will permanently be endowed with this topology throughout the paper. 

\subsection{Skeleton equations} \label{subsec: skeleton equation}

In this subsection, we will consider the skeleton equations \eqref{eq: skeleton equation} and \eqref{eq: skeleton equation 0}, which give the rate functions corresponding to solution maps $\{ Y^{\epsilon}_{\cdot;t_0,\xi}\}_{\epsilon>0}$ and stationary solution $\{\mathcal{Y}^{\epsilon}\}_{\epsilon>0}$. Firstly, we will give the estimates of the skeleton equations.

\begin{lemma} \label{lem: well-posedness-skeleton-equation}
 For any $t_0 \in \R$, $\xi \in C_r$ and $v \in S_{M}$, Eq. \eqref{eq: skeleton equation} admits a unique solution map $Y^{0,v}_{t;t_0,\xi}$. For $\lambda \in(0,(2\lambda_{1}-2\lambda_{2}\mu_{1}^{(2r)}-\lambda_{3}\mu_{2}^{(2r)}) \wedge 2r)$, there is a constant $C(\lambda)>0$ such that
 \begin{equation} \label{eq: estimate-skeleton-1}
  \big\|Y^{0,v}_{t;t_0,\xi} \big\|_{r}^{2} \leq 2 \, e^{-\lambda (t-t_0)+2 M} \, \| \xi \|_{r}^2 + e^{2 M} \, C(\lambda) , \quad \forall \, t \geq t_0;
 \end{equation}
 and for any $\xi_1,\xi_2 \in C_r$, we have
 \begin{equation} \label{eq: estimate-skeleton-2}
  \big\|Y^{0,v}_{t;t_0,\xi_1}- Y^{0,v}_{t;t_0,\xi_2}\big\|_{r}^{2} \leq 2 \, e^{-\lambda (t-t_0)+ M}\, \| \xi_1-\xi_2 \|_{r}^2 .
 \end{equation}
 $\{Y^{0,v}_{\cdot; -n, \xi}\}_{n \in \N}$ is a Cauchy sequence in $\sC_r$ and the limit $\cY^{0,v,\ast}$ satisfies Eq. \eqref{eq: skeleton equation 0}.
\end{lemma}
\begin{proof}
 The proof of the local well-posedness is classical, we omit it. To get the global well-posedness of Eq. \eqref{eq: skeleton equation}, we need to give a priori estimate of $Y^{0,v}_{t;t_0,\xi}$. For $t \geq t_0$, we have
 \begin{align*}
  e^{\lambda (t-t_0)} |Y^{0,v}(t; t_0, \xi)|^2 & = |\xi(0)|^2+ \int_{t_0}^{t} e^{\lambda (s-t_0)} \big( \lambda \, |Y^{0,v}(s; t_0, \xi)|^{2}+ 2 \, Y^{0,v}(s; t_0, \xi)^{\prime} \cdot b(Y^{0,v}_{s;t_0,\xi}) \big) \, \d s \\
  & \quad + 2 \int_{t_0}^{t} e^{\lambda (s-t_0)} \, Y^{0,v}(s; t_0, \xi)^{\prime} \cdot \sigma(Y^{0,v}_{s;t_0,\xi}) \cdot v(s) \, \d s. 
 \end{align*}
 Following \cite[Theorem 3.2]{Wu_Stochastic_2017}, we apply Cauchy--Schwarz inequality to obtain 
 \begin{align*}
  & e^{\lambda (t-t_0)} |Y^{0,v}(t; t_0, \xi)|^2 \\
  & \leq |\xi(0)|^2+ \int_{t_0}^{t} e^{\lambda (s-t_0)} \Big( \lambda \, |Y^{0,v}(s; t_0, \xi)|^{2}+ 2 Y^{0,v}(s; t_0, \xi)^{\prime} \cdot \big( b(Y^{0,v}_{s;t_0,\xi})-b(0)\big) \Big) \, \d s \\
  & \quad + \int_{t_0}^{t} e^{\lambda (s-t_0)} \big( \epsilon_1 \, |Y^{0,v}(s; t_0, \xi)|^{2} + \frac{1}{\epsilon_1} |b(0)|^2 \big) \, \d s + 2 \int_{t_0}^{t} e^{\lambda (s-t_0)} \, |v(s)|^2 \, |Y^{0,v}(s; t_0, \xi)|^2 \, \d s \\
  & \quad + \int_{t_0}^{t} e^{\lambda (s-t_0)} \, \Big( |\sigma(Y^{0,v}_{s;t_0,\xi}) -\sigma(0)|^2+ |\sigma(0)|^2 \Big)\, \d s,
 \end{align*}
 where $0<\epsilon_1< 2\lambda_{1}-2\lambda_{2}\mu_{1}^{(2r)}-\lambda_{3}\mu_{2}^{(2r)}-\lambda $. Assumptions (H2-H3) and Lemma \ref{lem: estimate for mu} yield
 \begin{align*}
   e^{\lambda (t-t_0)} |Y^{0,v}(t; t_0, \xi)|^2 & \leq |\xi(0)|^2+ \int_{t_0}^{t} e^{\lambda (s-t_0)} \big(-2\lambda_{1}+ \lambda+2\lambda_{2}\mu_{1}^{(2r)}+\lambda_{3}\mu_{2}^{(2r)}+\epsilon_1 \big) \, |Y^{0,v}(s; t_0, \xi)|^{2}\, \d s \\
  & \quad + 2 \int_{t_0}^{t} e^{\lambda (s-t_0)} \, |v(s)|^2 \, |Y^{0,v}(s; t_0, \xi)|^2 \, \d s + \int_{t_0}^{t} e^{\lambda (s-t_0)} \, \Big( \frac{1}{\epsilon_1} |b(0)|^2+ |\sigma(0)|^2 \Big)\, \d s.
 \end{align*}
 Observe that the first term of the right hand of the above inequality is negative. Applying the Gr\"onwall inequality for $e^{\lambda (t-t_0)} |Y^{0,v}(t; t_0, \xi)|^2$, we obtain that it is less than
 \begin{equation*} 
   \Big( |\xi(0)|^2 + \frac{1}{\lambda} \big( e^{\lambda (t-t_0)}-1 \big) \, \big( \frac{1}{\epsilon_1} |b(0)|^2+ |\sigma(0)|^2 \big) \Big) \times \exp \Big\{ 2 \int_{t_0}^{t} |v(s)|^2 \, \d s \Big\},
 \end{equation*}
 for any $t \geq t_0$. Note that $\lambda<2r$, applying Proposition \ref{prop: cr} iii), we have
 \begin{equation*}
  \big\|Y^{0,v}_{t;t_0,\xi} \big\|_{r}^{2} \leq e^{-\lambda (t-t_0)} \, \| \xi \|_{r}^2 + \Big( e^{-\lambda (t-t_0)} \, | \xi(0) | + C(\lambda) \Big) \times \exp \Big\{ 2 \int_{t_0}^{t} |v(s)|^2 \, \d s \Big\}, 
 \end{equation*}
 where $C(\lambda) = \frac{1}{\lambda} \big( \frac{1}{\epsilon_1} |b(0)|^2+ |\sigma(0)|^2 \big) $. Using the fact $v \in S_M$, we get a priori estimate \eqref{eq: estimate-skeleton-1}. So we get the existence of the solution map to skeleton equation \eqref{eq: skeleton equation}. 
 
 For any $\xi_1,\xi_2 \in C_r$, let $\eta(t)=Y^{0,v}(t; t_0, \xi_1)-Y^{0,v}(t; t_0, \xi_2) $, we have 
 \begin{align*}
  e^{\lambda (t-t_0)} |\eta(t)|^2 & = |\xi_1(0)-\xi_2(0)|^2+ \int_{t_0}^{t} e^{\lambda (s-t_0)} \big( \lambda \, |\eta(t)|^{2}+ 2 \eta(s)^{\prime} \cdot \big( b(Y^{0,v}_{s;t_0,\xi_1})-b(Y^{0,v}_{s;t_0,\xi_2}) \big) \big) \, \d s \\
  & \quad + 2 \int_{t_0}^{t} e^{\lambda (s-t_0)} \, \eta(s)^{\prime} \cdot \big( \sigma(Y^{0,v}_{s;t_0,\xi_1}) - \sigma(Y^{0,v}_{s;t_0,\xi_2})\big) \cdot v(s) \, \d s. 
 \end{align*}
 Similar to the proof estimate \eqref{eq: estimate-skeleton-1}, using assumptions (H2) and (H3), Cauchy--Schwarz inequality and Gr\"onwall inequality, we obtain estimate \eqref{eq: estimate-skeleton-2}. It implies the uniqueness of solution map $Y^{0,v}_{t;t_0,\xi}$. So we finished the proof of the well-posedness of skeleton Eq. \eqref{eq: skeleton equation}.
 
 Following the proof of Theorem \ref{thm: stationary solution}, the above estimates \eqref{eq: estimate-skeleton-1}, \eqref{eq: estimate-skeleton-2} and the dominated converge theorem yield that $\{Y^{0,v}_{\cdot; -n, \xi}\}_{n \in \N}$ is a Cauchy sequence in $\sC_r$ and the limit elements $\mathcal{Y}^{0,v,\ast}(t):= \lim_{n \rightarrow +\infty} Y^{0,v}_{t; -n, \xi} $ satisfies the skeleton equation \eqref{eq: skeleton equation 0}.
\end{proof}

Let ${\mathcal G}^{0}_{t_0,\xi} (\int_{0}^{\cdot} v(s) \d s)$ and ${\mathcal G}^{0}_{-\infty} (\int_{0}^{\cdot} v(s) \d s)$ be the solution to skeleton equations \eqref{eq: skeleton equation} and \eqref{eq: skeleton equation 0}, respectively.
Recall the definitions of $\cI_{t_0,\xi}$ and $\cI_{-\infty}$ in Section \ref{sec: Preliminary},
\begin{align*}
 \cI_{t_0,\xi} (\Phi) &:= \inf_{\big\{ v \in L^{2}(\R ; \R^m); \; \Phi ={\mathcal G}^{0}_{t_0,\xi} (\int_{0}^{\cdot} v(s) \d s) \big\}} \Big\{\frac{1}{2} \int_{-\infty}^{+\infty} |v(s)|^{2} \d s \Big\}; \\
 \cI_{-\infty}(\Phi) & := \inf_{\big\{ v \in L^{2}(\R ; \R^m); \; \Phi ={\mathcal G}^{0}_{-\infty} (\int_{0}^{\cdot} v(s) \d s) \big\}} \Big\{\frac{1}{2} \int_{-\infty}^{+\infty} |v(s)|^{2} \d s \Big\},
\end{align*}
where the infimum over the empty set is taken to be $+\infty$. We will show $\cI_{t_0,\xi}$ and $\cI_{-\infty}$ are good rate functions. Analogous to \cite[Lemma 3.4]{Liu_Large_2020} and \cite[Lemma 4.1]{Mo_Large_2013}, the key point is to prove that $G_{t_0,\xi}: v \mapsto Y^{0,v}_{\cdot;t_0,\xi}$ is continuous mapping from $S_M$ to $\sC_r$. This continuity also implies that both $\mathcal{G}^{0}_{t_0,\xi}$ and $\mathcal{G}^0$ are measurable map.

\begin{lemma} \label{lem: proposition-skeleton-equation-1}
 For each $M>0$, $\xi \in C_r$ and fixed $t_0 \in \R$, the map $G_{t_0,\xi}: v \mapsto Y^{0,v}_{\cdot;t_0,\xi}$ is continuous from $S_M$ to $\sC_r$. Consequently, the following sets 
 \begin{equation*}
  \Gamma_{M;t_0,\xi}:= \Big\{ \mathcal{G}^{0}_{t_0,\xi} (\int_{0}^{\cdot} v(s) \d s) \in \sC_r: \, v \in S_M \Big\}; \quad \Gamma_{M;-\infty} := \Big\{ \mathcal{G}^{0} (\int_{0}^{\cdot} v(s) \d s) \in \sC_r: \, v \in S_M \Big\},
 \end{equation*}
 are compact subsets of $\sC_r$. So $\cI_{t_0,\xi}(\Phi)$ and $\cI_{-\infty}(\Phi)$ are good rate functions. 
 Furthermore, for each $t_0 \in \R$ and compact set $K \subset C_r$, the set $\Gamma_{M;t_0,K}:=\cup_{\xi \in K} \Gamma_{M;t_0,\xi}$ is compact. 
\end{lemma}
\begin{proof}
 Firstly, we will prove the map $G_{t_0,\xi}: v \mapsto Y^{0,v}_{\cdot;t_0,\xi}$ is continuous from $S_M$ to $\sC_r$. Let $\{ v_{n}\}_{n} \subset S_{M}$ be the sequence, and $\{v_{n}\}_{n}$ converge to $v $ in $ S_M$ as $n \rightarrow +\infty$. Then 
 \begin{align*}
  e^{\lambda (t-t_0)} \big| \eta^{(n)}(t) \big|^2 & = \int_{t_0}^{t} e^{\lambda (s-t_0)} \big( \lambda \, |\eta^{(n)}(s)|^{2}+ 2 \, \eta^{(n)}(s)^{\prime} \cdot \big( b(Y^{0,v_{n}}_{s;t_0,\xi})-b(Y^{0,v}_{s;t_0,\xi}) \big) \big) \, \d s \\
  & \quad + 2 \int_{t_0}^{t} e^{\lambda (s-t_0)} \, \eta^{(n)}(s)^{\prime} \cdot \big( \sigma(Y^{0,v_{n}}_{s;t_0,\xi}) - \sigma(Y^{0,v}_{s;t_0,\xi})\big) \cdot v_{n}(s) \, \d s \\
  & \quad + 2 \int_{t_0}^{t} e^{\lambda (s-t_0)} \, \eta^{(n)}(s)^{\prime} \cdot \sigma(Y^{0,v}_{s;t_0,\xi}) \cdot (v_{n}(s)-v(s)) \, \d s,
 \end{align*}
 where $\eta^{(n)}(t) := Y^{0,v_{n}}(t;t_0,\xi)-Y^{0,v}(t;t_0,\xi) $ and $\lambda \in(0,(2\lambda_{1}-2\lambda_{2}\mu_{1}^{(2r)}-\lambda_{3}\mu_{2}^{(2r)}) \wedge 2r)$. Let $R_n(t)$ be the third term in the right-hand of the above identity. By Cauchy--Schwarz inequality, 
 \begin{align*}
  e^{\lambda (t-t_0)} \big| \eta^{(n)}(t) \big|^2 & \leq \int_{t_0}^{t} e^{\lambda (s-t_0)} \big( \lambda \, |\eta^{(n)}(s)|^{2}+ 2 \, \eta^{(n)}(s)^{\prime} \cdot \big( b(Y^{0,v_{n}}_{s;t_0,\xi})-b(Y^{0,v}_{s;t_0,\xi}) \big) \big) \, \d s \\
  & \quad + \int_{t_0}^{t} e^{\lambda (s-t_0)} \, \Big(\big| \sigma(Y^{0,v_{n}}_{s;t_0,\xi}) - \sigma(Y^{0,v}_{s;t_0,\xi})\big|^2 + |\eta^{(n)}(s) |^2 |v_{n}(s)|^2 \Big) \, \d s +R_n(t). 
 \end{align*}
 Using assumptions (H2) and (H3), and applying Gr\"onwall inequality, we obtain 
 \begin{equation*}
  e^{\lambda (t-t_0)} \big| \eta^{(n)}(t) \big|^2 \leq \exp \Big\{ \int_{t_0}^{t} |v_n(s)|^2 \, \d s \Big\} \times \sup_{t_0 <s<t} |R_n(s)| \leq e^{M} \sup_{t_0 <s<t} |R_n(s)|.
 \end{equation*}
 Note that $\lambda<2r$, applying Proposition \ref{prop: cr} iii), we have
 \begin{equation*}
  \big\|Y^{0,v_{n}}_{t;t_0,\xi}-Y^{0,v}_{t;t_0,\xi} \big\|_{r}^{2} \leq e^{-\lambda (t-t_0)+M} \,\sup_{t_0 <s<t} |R_n(s)|.
 \end{equation*}
 Recall the definition of metric $d_r$, to prove the convergence of $\{G_{t_0,\xi}(v_n) := Y^{0,v_n}_{\cdot;t_0,\xi}\}_{n \geq 1}$ in $\sC_r$, we only need to prove $\lim\limits_{n \rightarrow +\infty} \, \sup_{s \in [t_0,t]} | R_n(s) |=0$. Let
 \begin{equation*}
  \zeta^{(n,\lambda)}(t) := 2 \int_{t_0}^{t} e^{\lambda (s-t_0)} \, \sigma(Y^{0,v}_{s;t_0,\xi}) \cdot (v_{n}(s)-v(s)) \, \d s, \quad \forall \, t \geq t_0.
 \end{equation*}
 Due to $|\sigma(Y^{0,v}_{s;t_0,\xi})|$ is uniform bounded and the fact that $v_n$ converge to $v$ in $S_M$, using the Arz\'ela--Ascoli theorem, we obtain $\zeta^{(n,\lambda)}$ converge to zero in $C([t_0,T]; \R^{d})$. It implies that
 \begin{equation*}
  \lim_{n \rightarrow +\infty} \, \sup_{t \in [t_0,T]} | \zeta^{(n,\lambda)}(t) |=0, \quad \forall \, T >t_0.
 \end{equation*} 
 By Newton-Leibniz formula, we have
 \begin{align*}
  R_n(t) & = \eta^{(n)}(t)^{\prime} \cdot \zeta^{(n,\lambda)}(t)-\int_{t_0}^{t} \big( b(Y^{0,v_{n}}_{s;t_0,\xi})-b(Y^{0,v}_{s;t_0,\xi}) \big)^{\prime} \cdot \zeta^{(n,\lambda)}(s) \, \d s \\
  & \quad - \int_{t_0}^{t} \big( \sigma(Y^{0,v_{n}}_{s;t_0,\xi}) \cdot v^{n}(s)-\sigma(Y^{0,v}_{s;t_0,\xi}) \cdot v(s) \big)^{\prime} \cdot \zeta^{(n,\lambda)}(s) \, \d s.
 \end{align*}
Applying assumptions (H1), (H3) and estimate \eqref{eq: estimate-skeleton-1}, we obtain that $b(Y^{0,v_{n}}_{t;t_0,\xi})$ and $\sigma(Y^{0,v_{n}}_{t;t_0,\xi})$ is uniform bounded. Thus, the convergence of $\zeta^{(n,\lambda)}$ implies $\lim\limits_{n \rightarrow +\infty} \, \sup_{s \in [t_0,t]} | R_n(s) |=0$. As the above discussions, we get that the map $G_{t_0,\xi}$ is continuous from $S_M$ to $\sC_r$. By the compactness of $S_M$ and the continuity of $G_{t_0,\xi}$, we obtain that $\Gamma_{M;t_0,\xi}:= \big\{ \mathcal{G}^{0}_{t_0,\xi} (\int_{0}^{\cdot} v(s) \d s): \, v \in S_M \big\}$ is compact. For each compact set $K \subset C_r$, using estimate \eqref{eq: estimate-skeleton-2} and the compactness of $\Gamma_{M;t_0,\xi}$, we get that $\Gamma_{M;t_0,K}:=\cup_{\xi \in K} \Gamma_{M;t_0,\xi}$ is also compact. Due to the level set
 \begin{equation*}
  \{\Phi \in \sC_r: \cI_{t_0,\xi}(\Phi) \leq M \}=\bigcap_{n \geq 1} \Gamma_{2M+\frac{1}{n};t_0,\xi}
 \end{equation*}
 is compact, the rate functions $\cI_{t_0,\xi}$ is a good rate function.

 Next, we will prove that $\Gamma_{M;-\infty} := \big\{ \mathcal{G}^{0} (\int_{0}^{\cdot} v(s) \d s) : \, v \in S_M \big\}$ is compact. By the dominated convergence theorem, we only need to prove that for each $t \in \R$, if $v_n $ convergence to $ v$ in $S_M$, then $\lim\limits_{n \rightarrow +\infty} \| \mathcal{Y}^{0,v_n,\ast}(t)-\mathcal{Y}^{0,v,\ast}(t)\|_r = 0$. Due to $\mathcal{Y}^{0,v_n,\ast}$ and $\mathcal{Y}^{0,v_n,\ast}$ satisfy \eqref{eq: skeleton equation 0}, we obtain
 \begin{align*}
  & \| \mathcal{Y}^{0,v_n,\ast}(t)-\mathcal{Y}^{0,v,\ast}(t)\|_r =\| Y^{0,v_n}_{t;-N,\mathcal{Y}^{0,v_n,\ast}(-N)}-Y^{0,v}_{t;-N,\mathcal{Y}^{0,v,\ast}(-N)}\|_r \\
  & \quad \leq \| Y^{0,v_n}_{t;-N,\mathcal{Y}^{0,v_n,\ast}(-N)}-Y^{0,v_n}_{t;-N,\mathcal{Y}^{0,v,\ast}(-N)}\|_r +\| Y^{0,v_n}_{t;-N,\mathcal{Y}^{0,v,\ast}(-N)}-Y^{0,v}_{t;-N,\mathcal{Y}^{0,v,\ast}(-N)}\|_r,
 \end{align*}
 where $N>|t|$. Applying the result that $\Gamma_{M;-N,\mathcal{Y}^{0,v,\ast}(-N)}$ is compact, we get
 \begin{equation*}
  \lim_{n \rightarrow +\infty} \| \mathcal{Y}^{0,v_n,\ast}(t)-\mathcal{Y}^{0,v,\ast}(t)\|_r \leq \lim_{n \rightarrow +\infty} \sup_{n \geq 1} \| Y^{0,v_n}_{t;-N,\mathcal{Y}^{0,v_n,\ast}(-N)}-Y^{0,v_n}_{t;-N,\mathcal{Y}^{0,v,\ast}(-N)}\|_r.
 \end{equation*}
 By estimate \eqref{eq: estimate-skeleton-2}, we have 
 \begin{equation*}
  \sup_{n \geq 1} \| Y^{0,v_n}_{t;-N,\mathcal{Y}^{0,v_n,\ast}(-N)}-Y^{0,v_n}_{t;-N,\mathcal{Y}^{0,v,\ast}(-N)}\|_r \leq 2 \, e^{-\lambda (t+N)+ M}\, \sup_{n \geq 1} \| \mathcal{Y}^{0,v_n,\ast}(-N)-\mathcal{Y}^{0,v,\ast}(-N) \|_{r}^2.
 \end{equation*}
Following from the estimates \eqref{eq: estimate-skeleton-1} and \eqref{eq: estimate-skeleton-2}, we know that $\sup_{n,N}\|\mathcal{Y}^{0,v_n,\ast}(-N)\|_r<+\infty$. Combining the above estimates, taking $N \rightarrow + \infty$, we get $\lim\limits_{n \rightarrow +\infty} \| \mathcal{Y}^{0,v_n,\ast}(t)-\mathcal{Y}^{0,v,\ast}(t)\|_r = 0$. So the set $\Gamma_{M;-\infty}$ is compact, and $\cI_{-\infty}$ is also a good rate function.
\end{proof}

To obtain the uniform LDP, we need the lower semicontinuity of map $\xi \mapsto \cI_{t_0,\xi}(\Phi)$ and the compactness of $\Lambda_{M,t_0,K}= \cup_{\xi \in K} \{\Phi \in \sC_r; \, \cI_{t_0,\xi}(\Phi) \leq M\}$.

\begin{lemma} \label{lem: proposition-skeleton-equation-2}
For each $\Phi \in \sC_r$, $\xi \mapsto \cI_{t_0,\xi}(\Phi)$ is a lower semicontinuous (l.s.c.) map from $ C_r$ to $[0,+\infty]$. Furthermore, for each $M<+\infty$,
 \begin{equation*}
  \Lambda_{M,t_0,K} := \big\{ \Phi \in \sC_r: \,\, \cI_{t_0,\xi}(\Phi) \leq M, \, \xi \in K \big\}
 \end{equation*}
 is a compact subset of $\sC_r$. 
\end{lemma}
\begin{proof}
 Let $\{ \xi_{n} \}_{n \in \N}$ be a convergence sequence in $C_r$ with the limit $\xi$. For each $\Phi \in \sC_r$, if 
 \begin{equation*}
  \liminf_{n \rightarrow +\infty} \cI_{t_0,\xi_n}(\Phi)=M <+\infty,
 \end{equation*}
 then by the definition of $\cI_{t_0,\xi_n}(\Phi)$, for any $\delta>0$, there are $\{ v_{n_k} \}_{k \in \N} \subset S_{2M+\delta}$ such that
 \begin{equation*}
  \Phi ={\mathcal G}^{0}_{t_0,\xi_{n_k}} \big(\int_{0}^{\cdot} v_{n_k}(s) \d s \big), \quad \forall \, k \in \N,
 \end{equation*}
 where $\{n_k\}_{k}$ is a subsequence of $\N$. By the compactness of $S_{2M+\delta}$, there is a subsequence (we also denote it as $\{n_k\}_{k}$) such that $v_{n_{k}}$ converge to $v \in S_{2M+\delta}$ as $k \rightarrow +\infty$. Then by the continuity of map $G_{t_0,\xi}: v \mapsto Y^{0,v}_{\cdot;t_0,\xi}$ and estimate \eqref{eq: estimate-skeleton-2}, for any $t_0 \in \R$, we have
 \begin{align*}
  \Big\| \Phi(t)-{\mathcal G}^{0}_{t_0,\xi} \big(\int_{0}^{\cdot} v(s) \d s \big)(t) \Big\|_r & =\lim_{k \rightarrow +\infty} \Big\| {\mathcal G}^{0}_{t_0,\xi_{n_k}} \big(\int_{0}^{\cdot} v_{n_k}(s) \d s \big)(t)-{\mathcal G}^{0}_{t_0,\xi} \big(\int_{0}^{\cdot} v(s) \d s \big)(t) \Big\|_r \\
  & \leq \lim_{k \rightarrow +\infty} \Big\| {\mathcal G}^{0}_{t_0,\xi_{n_k}} \big(\int_{0}^{\cdot} v_{n_k}(s) \d s \big)(t) - {\mathcal G}^{0}_{t_0,\xi} \big(\int_{0}^{\cdot} v_{n_k}(s) \d s \big)(t) \Big\|_r \\
  & \quad + \lim_{k \rightarrow +\infty} \Big\| {\mathcal G}^{0}_{t_0,\xi} \big(\int_{0}^{\cdot} v_{n_k}(s) \d s \big)(t)-{\mathcal G}^{0}_{t_0,\xi} \big(\int_{0}^{\cdot} v(s) \d s \big)(t) \Big\|_r =0.
 \end{align*}
 So we get ${\mathcal G}^{0}_{t_0,\xi} \big(\int_{0}^{\cdot} v(s) \d s \big)=\Phi$, then $\cI_{t_0,\xi}(\Phi) \leq M+\delta/2$. By the arbitrariness of $\delta$, we have 
 \begin{equation*}
  \cI_{t_0,\xi}(\Phi) \leq M= \liminf_{n \rightarrow +\infty} \cI_{t_0,\xi_n}(\Phi).
 \end{equation*}
 We get $\xi \mapsto \cI_{t_0,\xi}(\Phi)$ is an l.s.c. map. Then following the proof of \cite[Theorem 5]{Budhiraja_Large_2008}, we obtain
 \begin{equation*}
  \Lambda_{M,t_0,K}= \bigcap_{n \geq 1} \Gamma_{2M+\frac{1}{n};t_0,K}.
 \end{equation*}
 Thus, $\Lambda_{M,t_0,K}$ is compact follows from the compactness of $\Gamma_{2M+\frac{1}{n};t_0,K}$.
\end{proof}

\subsection{Uniform Large Deviation Principle for solution map} \label{subsec: uniform LDP}

In this subsection, we will prove Theorem \ref{thm: uniform-LDP}, the uniform LDP for $\{ Y^{\epsilon}_{\cdot;t_0,\xi}\}_{\epsilon>0}$. Due to Lemma \ref{lem: truncated-measurable-map}, we have $\P$-a.s. $Y^{\epsilon}_{\cdot;t_0,\xi}= \mathcal{G}_{t_0}^{\epsilon}(W,\xi)=:\mathcal{G}_{t_0,\xi}^{\epsilon}(W)$. Following \cite[Theorem 5]{Budhiraja_Large_2008}, using the variational representation for double-side Wiener process (Lemma \ref{lem: variational representation}) and the weak converge theory, we know that the following conditions are sufficient for the uniform LDP for $\{ Y^{\epsilon}_{\cdot;t_0,\xi}\}_{\epsilon>0}$.

\begin{itemize}
 \item [(A1)] Let $M >0$, families $\{v_{\epsilon}: \,\epsilon > 0\} \subset{\cA}_{M}$ and $\{\xi^{\epsilon} \} \subset C_r$. If $v_{\epsilon}$ converge to $v$ in distribution as $S_M$-valued random elements and $\xi^{\epsilon} \rightarrow \xi$ in $C_r$ as $\epsilon \rightarrow 0$, then
 \begin{equation*}
  \mathcal{G}_{t_0; \xi^{\epsilon}}^{\epsilon} \Big(W(\cdot)+\frac{1}{\sqrt{\epsilon}} \int_{0}^{\cdot} v_{\epsilon}(s) \, \d s\Big) \rightarrow \mathcal{G}_{t_0,\xi}^{0} \Big(\int_{0}^{\cdot} v(s) \, \d s\Big) 
 \end{equation*}
 in distribution as $\epsilon \rightarrow 0$. 
 \item [(A2)] $\xi \mapsto \cI_{t_0,\xi}$ is an l.s.c. map from $C_r$ to $[0,+\infty]$; and for each $M>0$ and compact set $K \subset C_r$, the set $\Gamma_{M;t_0,K}=\big\{ \mathcal{G}^{0}_{t_0,\xi} (\int_{0}^{\cdot} v(s) \d s) \in \sC_r: \, v \in S_M, \, \xi \in K \big\} $ is compact.
\end{itemize}

We have proved the condition (A2) in Lemma \ref{lem: proposition-skeleton-equation-1} and \ref{lem: proposition-skeleton-equation-2}. Next, we will verify the condition (A1). For each $\xi \in C_r$ and $v \in \cA_M$, let $Y_{t;t_0,\xi}^{\epsilon,v}$ be the solution of 
\begin{equation} \label{eq: truncated-controlled-equation}
 Y^{\epsilon,v}_{t; t_0, \xi}(0)= \xi(0)+\int_{t_0}^{t} b(Y^{\epsilon,v}_{s;t_0,\xi}) \, \d s + \int_{t_0}^{t} \sigma(Y^{\epsilon,v}_{s;t_0,\xi}) \cdot v(s) \, \d s + \sqrt{\epsilon} \int_{t_0}^{t} \sigma(Y^{\epsilon,v}_{s;t_0,\xi}) \, \d W(s) ,
\end{equation}
if $t \geq t_0$ and $\xi^{\epsilon}(t-t_0)$ otherwise, and $Y^{\epsilon,v_{\epsilon}}_{t; t_0, \xi^{\epsilon}}(\tau) =Y^{\epsilon,v_{\epsilon}}_{t+\tau; t_0, \xi^{\epsilon}}(0)$ for any $\tau \in (-\infty,0)$. By Lemma \ref{lem: truncated-measurable-map} and the Girsanov theorem, we know that for each $\epsilon>0$,
\begin{equation*}
 \P\mbox{-}a.s., \quad Y^{\epsilon,v_{\epsilon}}_{t;t_0,\xi^{\epsilon}}=\mathcal{G}^{\epsilon}_{t_0;\xi^{\epsilon}} \Big(W(\cdot)+\frac{1}{\sqrt{\epsilon}} \int_{0}^{\cdot} v_{\epsilon}(s) \, \d s\Big)(t), \quad \forall \, t \in \R,
\end{equation*}
which also implies the existence of solution $Y_{t;t_0,\xi}^{\epsilon,v}$ to the controlled Eq. \eqref{eq: truncated-controlled-equation}. In order to verify the condition (A1), we need the following useful estimates of $Y_{t;t_0,\xi}^{\epsilon,v}$. To prove these estimates, we will use the stochastic Gr\"onwall inequality, see in Lemma \ref{lem: Stochastic Gronwall}. 
\begin{lemma} \label{lem: truncated-controlled-equation}
 For each $M>0$, $v \in \cA_M $ and $\xi \in C_r$, Eq. \eqref{eq: truncated-controlled-equation} admits a unique solution $Y^{\epsilon,v}_{t; t_0, \xi}$ for any $0< \epsilon < \epsilon_0:= \big((2 \lambda_1-2\lambda_2 \mu_{1}^{(2r)})/ (\lambda_3 \mu_{2}^{(2r)})-1 \big) \wedge 1$. For any $\lambda \in (0, 2 \lambda_1- 2 \lambda_2 \mu_{1}^{(2r)} - (1+\epsilon) \lambda_3 \mu_{2}^{(2r)} \wedge 2r)$, we have the following estimates.
 \begin{itemize}
  \item [i)] If $\epsilon_1>0$ and $\epsilon_2>0$ are satisfying
  \begin{equation*}
   2 \lambda_1-\epsilon_1 -2 \lambda_2 \mu_{1}^{(2r)} - \frac{(1+\epsilon)}{1-\epsilon_2} \lambda_3 \mu_{2}^{(2r)}>\lambda >0,
  \end{equation*} 
  then for any $t \geq t_0$, we have 
  \begin{align} 
   \E \| Y^{\epsilon,v}_{t;t_0,\xi}\|_{r} & \leq \Tilde{C}_1(\epsilon,\lambda) \, e^{ \frac{-\lambda (t-t_0)}{2}+\frac{M}{2(1-\epsilon_2)}} \, \|\xi\|_{r} + \Tilde{C}_2(\epsilon,\lambda) \, e^{\frac{M}{2(1-\epsilon_2)}}; \label{eq: estimate-controlled-solution-1} \\
   \E \big[\| Y^{\epsilon,v}_{t;t_0,\xi}\|_{r}^{3/2} \big]^{2/3} & \leq 2 \Tilde{C}_1(\epsilon,\lambda) \, e^{ \frac{-\lambda (t-t_0)}{2}+\frac{M}{2(1-\epsilon_2)}} \, \|\xi\|_{r} + 2\Tilde{C}_2(\epsilon,\lambda) \, e^{\frac{M}{2(1-\epsilon_2)}}, \label{eq: estimate-controlled-solution-2}
  \end{align}
  where $ \Tilde{C}_1 $ and $\Tilde{C}_2$ defined as 
  \begin{align*}
   & \Tilde{C}_1(\epsilon,\lambda) := 1+ (3 \pi +1)^{2/3} \Big( 1 + \frac{2\lambda_{2}\mu_{1}^{(2r)} +\frac{(1+ \epsilon)}{1-\epsilon_2} \lambda_{3}\mu_{2}^{(2r)}}{2r-\lambda} \Big)^{1/2} ;\\
   & \Tilde{C}_2(\epsilon,\lambda) := \frac{(3 \pi +1)^{2/3}}{\sqrt{\lambda}} \Big( \frac{1}{\epsilon_1} |b(0)|^2 + \frac{1+\epsilon}{\epsilon_2} |\sigma(0)|^2 \Big)^{1/2}.
  \end{align*}
  Thus, $\{Y^{\epsilon,v}_{t;t_0,\xi}\}_{\epsilon,n,t}$ is uniform bounded in $L^{3/2} (\Omega, C_r)$ respect with to $\epsilon$ and $t_0,t \in \R$.

  \item [ii)] For any two $\mathcal{F}_{t_0}$-measurable random initial data $\xi_1, \xi_2 \in L^{3/2}(\Omega,C_r)$, we have
  \begin{equation} \label{eq: estimate-controlled-solution-diff}
   \E \| Y^{\epsilon,v}_{t;t_0,\xi_1} -Y^{\epsilon,v}_{t;t_0,\xi_2} \|_{r} \leq \Tilde{C}_3(\epsilon,\lambda) \, e^{ \frac{-\lambda (t-t_0)+M}{2}} \, \E \big[ \|\xi_1-\xi_2\|_{r}^{3/2} \big]^{2/3}.
  \end{equation}
  where constant $\Tilde{C}_2(\epsilon,\lambda)$ defined as 
  \begin{equation*}
   \Tilde{C}_3(\epsilon,\lambda) := 1+ (3 \pi +1)^{2/3} \Big( 1 + \frac{2\lambda_{2}\mu_{1}^{(2r)} +(1+ \epsilon) \lambda_{3}\mu_{2}^{(2r)}}{2r-\lambda} \Big)^{1/2}.
  \end{equation*}
  Thus, for each initial data $\xi \in C_r$, $\{Y^{\epsilon,v}_{t;-n,\xi}\}_{n \geq 1}$ is a Cauchy sequence in $L^1(\Omega,C_r)$.
 \end{itemize}
\end{lemma}

\begin{proof}
 Firstly, we proceed to prove Lemma \ref{lem: truncated-controlled-equation} ii). For two $\mathcal{F}_{t_0}$-measurable random initial data $\xi_1, \xi_2 \in L^{3/2}(\Omega,C_r)$ and $v \in \cA_M$, let $\eta(t):= Y^{\epsilon,v}(t;t_0,\xi_1) - Y^{\epsilon,v}(t;t_0,\xi_2)$. Then we obtain
 \begin{align}
  e^{\lambda (t-t_0)} |\eta(t)|^2 & = |\eta(t_0)|^2 + \int_{t_0}^{t} e^{\lambda (s-t_0)} \left( \lambda |\eta(s)|^{2}+ 2 \eta(s)^{\prime} \Big( b(Y^{\epsilon,v}_{s;t_0,\xi_1})-b(Y^{\epsilon,v}_{s;t_0,\xi_2}) \Big) \right) \, \d s \nonumber \\
  & \quad + 2 \int_{t_0}^{t} e^{\lambda (s-t_0)} \, \eta(s)^{\prime} \, v(s) \, \left( \sigma(Y^{\epsilon,v}_{s;t_0,\xi_1})- \sigma(Y^{\epsilon,v}_{s;t_0,\xi_2}) \right) \, \d s \nonumber \\
  & \quad + \epsilon \int_{t_0}^{t} e^{\lambda (s-t_0)} \, \big|\sigma(Y^{\epsilon,v}_{s;t_0,\xi_1})-\sigma(Y^{\epsilon,v}_{s;t_0,\xi_2}) \big|^{2} \, \d s \nonumber \\
   & \quad + 2 \sqrt{\epsilon} \int_{t_0}^{t} e^{\lambda (s-t_0)} \, \eta(s)^{\prime} \left(\sigma(Y^{\epsilon,v}_{s;t_0,\xi_1})-\sigma(Y^{\epsilon,v}_{s;t_0,\xi_2}) \right) \, \d W(s). \label{eq: energy-controlled}
 \end{align}
 Except for the third term on the right-hand side of \eqref{eq: energy-controlled}, the estimates for all other items can be found on the \cite[Theorem 3.2]{Wu_Stochastic_2017}. By Cauchy--Schwarz inequality, we obtain 
 \begin{align*}
  & \int_{t_0}^{t} e^{\lambda (s-t_0)} \, \eta(s)^{\prime} \, v(s) \, \left( \sigma(Y^{\epsilon,v}_{s;t_0,\xi_1})- \sigma(Y^{\epsilon,v}_{s;t_0,\xi_2}) \right) \, \d s \\
  & \quad \leq \frac{1}{2} \int_{t_0}^{t} e^{\lambda (s-t_0)} |v(s)|^2 |\eta(s)|^2 \, \d s + \frac{1}{2} \int_{t_0}^{t} e^{\lambda (s-t_0)} \, \big|\sigma(Y^{\epsilon,v}_{s;t_0,\xi_1})-\sigma(Y^{\epsilon,v}_{s;t_0,\xi_2}) \big|^{2} \, \d s.
 \end{align*}
 Analogous to the proof of \cite[Theorem 3.2]{Wu_Stochastic_2017}, assumption (H3) and Lemma \ref{lem: estimate for mu} yields
 \begin{align*}
  & \int_{t_0}^{t} e^{\lambda (s-t_0)} \, \big|\sigma(Y^{\epsilon,v}_{s;t_0,\xi_1})-\sigma(Y^{\epsilon,v}_{s;t_0,\xi_2}) \big|^{2} \, \d s \\
  & \quad \leq \lambda_3 \int_{t_0}^{t} e^{\lambda (s-t_0)} \int_{-\infty}^{0} \big|Y^{\epsilon,v}_{s;t_0,\xi_1}(\tau)-Y^{\epsilon,v}_{s;t_0,\xi_2} (\tau) \big| \, \mu_2^{(2r)}(\d \tau) \\
   & \quad \leq \frac{\lambda_3 \mu_{2}^{(2r)}}{2r-\lambda} \|\xi_1-\xi_2\|_{r}^2 + \lambda_3 \mu_{2}^{(2r)} \int_{t_0}^{t} e^{\lambda (s-t_0)} \vert \eta(s)|^{2} \, \d s.
 \end{align*}
 Similarly, by assumption (H2) and Lemma \ref{lem: estimate for mu}, we obtain
 \begin{align*}
  & 2 \int_{t_0}^{t} e^{\lambda (s-t_0)} \, \eta(s)^{\prime} \Big( b(Y^{\epsilon,v}_{s;t_0,\xi_1})-b(Y^{\epsilon,v}_{s;t_0,\xi_2}) \Big) \, \d s \\
  & \quad \leq \big(- 2 \lambda_1+2 \lambda_2 \mu_{1}^{(2r)} \big) \int_{t_0}^{t} e^{\lambda (s-t_0)} |\eta(s)|^{2} \, \d s + \frac{2 \lambda_2 \mu_{1}^{(2r)}}{2r-\lambda} \|\xi_1-\xi_2\|_{r}^2.
 \end{align*}
 Combining the above estimates, from \eqref{eq: energy-controlled} we get 
 \begin{align}
  e^{\lambda (t-t_0)} |\eta(t)|^2 & \leq |\eta(0)|^2 + \frac{2\lambda_{2}\mu_{1}^{(2r)} +(1+ \epsilon) \lambda_{3}\mu_{2}^{(2r)}}{2r-\lambda} \|\xi_1-\xi_2\|_{r}^2 +\int_{t_0}^{t} e^{\lambda (s-t_0)} |v(s)|^2 |\eta(s)|^2 \, \d s \nonumber \\
  & \quad + \big(- 2 \lambda_1+2 \lambda_2 \mu_{1}^{(2r)} + (1+\epsilon) \lambda_3 \mu_{2}^{(2r)} +\lambda \big) \int_{t_0}^{t} e^{\lambda (s-t_0)} |\eta(s)|^{2} \, \d s \nonumber \\
  & \quad + 2 \sqrt{\epsilon} \int_{t_0}^{t} e^{\lambda (s-t_0)} \, \eta(s)^{\prime} \left(\sigma(Y^{\epsilon,v}_{s;t_0,\xi_1})-\sigma(Y^{\epsilon,v}_{s;t_0,\xi_2}) \right) \, \d W(s). \label{eq: energy-controlled-1}
 \end{align}
 Due to $0<\lambda <2 \lambda_1- 2 \lambda_2 \mu_{1}^{(2r)} - (1+\epsilon) \lambda_3 \mu_{2}^{(2r)}$, the fourth term on the right side of \eqref{eq: energy-controlled-1} is less than zero. Applying Lemma \ref{lem: Stochastic Gronwall} with $p=1/2$, $p_1=3/2$ and $p_2=3$, we obtain
 \begin{align*}
  \E \sup_{t_0 \leq s \leq t} e^{\lambda (s-t_0)/2} |\eta(s)| & \leq (3 \pi +1)^{2/3} \Big( \E \exp \Big\{\frac{3}{2} \int_{t_0}^{t} |v(s)|^2 \d s \Big\} \Big)^{1/3} \\
  & \quad \times \E \bigg[ \Big( |\eta(0)|^2 + \frac{2\lambda_{2}\mu_{1}^{(2r)} +(1+ \epsilon) \lambda_{3}\mu_{2}^{(2r)}}{2r-\lambda} \|\xi_1-\xi_2\|_{r}^{2} \Big)^{3/4} \bigg]^{2/3}.
 \end{align*}
 By the fact that $v \in \cA_M$, we have 
 \begin{equation*}
  \E \sup_{t_0 \leq s \leq t} e^{\lambda (s-t_0)/2} |\eta(s)| \leq (3 \pi +1)^{2/3} \Big( 1 + \frac{2\lambda_{2}\mu_{1}^{(2r)} +(1+ \epsilon) \lambda_{3}\mu_{2}^{(2r)}}{2r-\lambda} \Big)^{1/2} e^{M/2} \big(\E \|\xi_1-\xi_2\|_{r}^{3/2} \big)^{2/3}.
 \end{equation*}
 Noting that $\lambda < 2r$, by Proposition \ref{prop: cr} iii) and triangle inequality, we obtain
 \begin{equation*}
  \E \| \eta_t \|_r = \E \sup_{-\infty < \tau \leq 0} e^{r \tau} \, |\eta(t+\tau)| \leq e^{-\lambda t/2} \, \|\xi_1-\xi_2\|_{r} + e^{-\lambda t/2} \, \E \sup_{t_0 < s \leq t } e^{\lambda s/2} \, |\eta(s)|.
 \end{equation*}
Therefore, combining the above estimates, we get \eqref{eq: estimate-controlled-solution-diff} and the uniqueness of solution $Y^{\epsilon,v}_{t; t_0, \xi}$.
 
 Now, we turn to prove Lemma \ref{lem: truncated-controlled-equation} i). Analogous to the proof of Lemma \ref{lem: truncated-controlled-equation} ii), using the estimates in \cite[Theorem 3.2]{Wu_Stochastic_2017} and choosing $\epsilon_1$ and $\epsilon_2$ suitable, we obtain
 \begin{align}
  & e^{\lambda (t-t_0)} |Y^{\epsilon,v}(t;t_0,\xi)|^2 -|Y^{\epsilon,v}(t;t_0,\xi)(0)|^2 \nonumber \\
  & \leq \frac{2\lambda_{2}\mu_{1}^{(2r)} +\frac{(1+ \epsilon)}{1-\epsilon_2} \lambda_{3}\mu_{2}^{(2r)}}{2r-\lambda} \|\xi\|_{r}^2 + \frac{1}{\lambda} \Big( \frac{1}{\epsilon_1} |b(0)|^2 + \frac{1+\epsilon}{\epsilon_2} |\sigma(0)|^2 \Big) \big(e^{\lambda(t-t_0)}-1 \big) \nonumber \\
  & \quad + \frac{1}{1-\epsilon_2}\int_{t_0}^{t} e^{\lambda (s-t_0)} |v(s)|^2 |Y^{\epsilon,v}(s;t_0,\xi)|^2 \, \d s \nonumber \\
  & \quad + \big(- 2 \lambda_1+ \epsilon_1 +2 \lambda_2 \mu_{1}^{(2r)} + \frac{(1+\epsilon)}{1-\epsilon_2} \lambda_3 \mu_{2}^{(2r)} +\lambda \big) \int_{t_0}^{t} e^{\lambda (s-t_0)} |Y^{\epsilon,v}(s;t_0,\xi)|^{2} \, \d s \nonumber \\
  & \quad + 2 \sqrt{\epsilon} \int_{t_0}^{t} e^{\lambda (s-t_0)} \, Y^{\epsilon,v}(s;t_0,\xi)^{\prime} \sigma(Y^{\epsilon,v}_{s;t_0,\xi}) \, \d W(s). \label{eq: energy-controlled-2}
 \end{align}
 Due to $0<\lambda<2 \lambda_1-\epsilon_1 -2 \lambda_2 \mu_{1}^{(2r)} - \frac{(1+\epsilon)}{1-\epsilon_2} \lambda_3 \mu_{2}^{(2r)} $, applying Lemma \ref{lem: Stochastic Gronwall} with $p=1/2$, $p_1=3/2$ and $p_2=3$ and using the fact that $v \in \cA_M$, we obtain
 \begin{align*}
  & \E \sup_{t_0 \leq s \leq t} e^{\lambda (s-t_0)/2} |Y^{\epsilon,v}(t;t_0,\xi)| \\
  & \leq (3 \pi +1)^{2/3} e^{\frac{M}{2(1-\epsilon_2)}} \Big( 1 + \frac{2\lambda_{2}\mu_{1}^{(2r)} +\frac{(1+ \epsilon)}{1-\epsilon_2} \lambda_{3}\mu_{2}^{(2r)}}{2r-\lambda} \Big)^{1/2} \|\xi\|_{r} \\
   & \quad + (3 \pi +1)^{2/3} e^{\frac{M}{2(1-\epsilon_2)}} \Big( \frac{1}{\lambda} \Big( \frac{1}{\epsilon_1} |b(0)|^2 + \frac{1+\epsilon}{\epsilon_2} |\sigma(0)|^2 \Big) e^{\lambda(t-t_0)} \Big)^{1/2}.
 \end{align*}
 Using Proposition \ref{prop: cr} iii) and triangle inequality, we can show \eqref{eq: estimate-controlled-solution-1}. 
 
 To obtain the higher moment estimate \eqref{eq: estimate-controlled-solution-2}, we only need to apply Lemma \ref{lem: Stochastic Gronwall} with $p=3/4$, $p_1=5/4$ and $p_2=5$ for inequality \eqref{eq: energy-controlled-2} and using the inequality $|a+b|^{q}\leq |a|^q+|b|^q$ for $q \in (0,1)$. Thus, $\{Y^{\epsilon,v}_{t;-n,\xi}\}_{\epsilon,n,t}$ is uniform bounded in $L^{3/2} (\Omega, C_r)$. Combining this uniform boundedness with \eqref{eq: estimate-controlled-solution-diff}, we obtain that $\{Y^{\epsilon,v}_{t;-n,\xi}\}_{n \geq 1}$ is a Cauchy sequence in $L^1(\Omega,C_r)$.
\end{proof}

We have established condition (A2) in Lemmas \ref{lem: proposition-skeleton-equation-1} and \ref{lem: proposition-skeleton-equation-2}, and we have prepared to address condition (A1). Now, it is time to prove the uniform LDP for the solution maps $\{ Y^{\epsilon}_{\cdot;t_0,\xi}\}_{\epsilon>0}$. 

\begin{proof}[\bf Proof of Theorem \ref{thm: uniform-LDP}]
 As the discussions at the beginning of Section \ref{subsec: uniform LDP}, to obtain the uniform LDP for $\{ Y^{\epsilon}_{t;t_0,\xi}\}_{\epsilon>0}$, we only need to verify the condition (A1) and (A2). In Lemma \ref{lem: proposition-skeleton-equation-1} and \ref{lem: proposition-skeleton-equation-2}, we have proved that the family of rate functions $\{ \cI_{t_0,\xi}: \, \xi \in C_r \}$ have compact level sets on compacts, and the condition (A2). Now, we focus on the condition (A1).

 Consider $M >0$, families $\{v_{\epsilon}: \,\epsilon > 0\} \subset{\cA}_{M}$ and $\{\xi^{\epsilon} \} \subset C_r$ such that $v_{\epsilon}$ converge to $v$ in distribution as $S_M$-valued random elements and $\xi^{\epsilon} \rightarrow \xi$ in $C_r$ as $\epsilon \rightarrow 0$. Let 
 \begin{equation*}
  Y_{\cdot;t_0,\xi^{\epsilon}}^{\epsilon,v_{\epsilon}}:= \mathcal{G}_{t_0; \xi^{\epsilon}}^{\epsilon} \Big(W(\cdot)+\frac{1}{\sqrt{\epsilon}} \int_{0}^{\cdot} v_{\epsilon}(s) \, \d s\Big).
 \end{equation*}
 Then by Lemma \ref{lem: truncated-measurable-map} and the Girsanov theorem, $Y_{t;t_0,\xi^{\epsilon}}^{\epsilon,v_{\epsilon}}$ is the solution of Eq. \eqref{eq: truncated-controlled-equation} with $v^\epsilon$ and initial data $\xi^{\epsilon}$.
 By the uniform boundedness given in Lemma \ref{lem: truncated-controlled-equation} i) and the definition of metric $d_r$, to prove the condition (A1) only need to prove that for each $t\geq t_0$, $Y_{t;t_0,\xi^{\epsilon}}^{\epsilon,v_{\epsilon}}$ converges in distribution to $Y^{0,v}_{t,t_0,\xi}$, where $Y^{0,v}_{t,t_0,\xi}$ is the solution of the skeleton equation \eqref{eq: skeleton equation}. 
 
 Considering the continuity of the map $v \in S_M \mapsto Y^{0,v}_{t,t_0,\xi} \in C_r$ as proven in Lemma \ref{lem: proposition-skeleton-equation-1} and the initial data dependence of the skeleton equation presented in Lemma \ref{lem: well-posedness-skeleton-equation}, we can infer that $Y^{0,v_{\epsilon}}_{t,t_0,\xi^{\epsilon}}$ converges in distribution to $Y^{0,v}_{t,t_0,\xi}$ as $\epsilon \rightarrow 0$. It remains to prove that $\eta^{\epsilon}(t) := Y_{t;t_0,\xi^{\epsilon}}^{\epsilon,v_{\epsilon}} - Y_{t;t_0,\xi^{\epsilon}}^{0,v_{\epsilon}}$ converges to zero in distribution. Employing the energy estimate, we know 
 \begin{align*}
  e^{\lambda (t-t_0)} |\eta^{\epsilon}(t)|^2 & = \int_{t_0}^{t} e^{\lambda (s-t_0)} \left( \lambda |\eta^{\epsilon}(s)|^{2}+ 2 \eta^{\epsilon}(s)^{\prime} \Big( b(Y^{\epsilon,v_{\epsilon}}_{s;t_0,\xi^{\epsilon}})-b(Y^{0,v_{\epsilon}}_{s;t_0,\xi^{\epsilon}}) \Big) \right) \, \d s \nonumber \\
  & \quad + 2 \int_{t_0}^{t} e^{\lambda (s-t_0)} \, \eta^{\epsilon}(s)^{\prime} \, v_{\epsilon}(s) \, \left( \sigma(Y^{\epsilon,v_{\epsilon}}_{s;t_0,\xi^{\epsilon}})- \sigma(Y^{0,v_{\epsilon}}_{s;t_0,\xi^{\epsilon}}) \right) \, \d s \nonumber \\
  & \quad + \epsilon \int_{t_0}^{t} e^{\lambda (s-t_0)} \, \big|\sigma(Y^{\epsilon,v_{\epsilon}}_{s;t_0,\xi^{\epsilon}}) \big|^{2} \, \d s + 2 \sqrt{\epsilon} \int_{t_0}^{t} e^{\lambda (s-t_0)} \, \eta^{\epsilon}(s)^{\prime} \sigma(Y^{\epsilon,v_{\epsilon}}_{s;t_0,\xi^{\epsilon}}) \, \d W(s),
 \end{align*}
 where $\lambda >0$ satisfies the same conditions in Lemma \ref{lem: truncated-controlled-equation}. Analogous to the proof of Lemma \ref{lem: truncated-controlled-equation} ii), using Lemma \ref{lem: Stochastic Gronwall}, we obtain
 \begin{align*}
  \E \sup_{t_0 \leq s \leq t} e^{\lambda (s-t_0)/2} |\eta^{\epsilon}(s)| & \leq (3 \pi +1)^{2/3} \Big( \E \exp \Big\{\frac{3}{2} \int_{t_0}^{t} |v_{\epsilon}(s)|^2 \d s \Big\} \Big)^{1/3} \\
  & \quad \times \sqrt{\epsilon} \, \E \bigg[ \Big| \int_{t_0}^{t} e^{\lambda (s-t_0)} \, \big|\sigma(Y^{\epsilon,v_{\epsilon}}_{s;t_0,\xi^{\epsilon}}) \big|^{2} \, \d s \Big|^{3/4} \bigg]^{2/3}.
 \end{align*}
 Using the condition (H3) and Cauchy--Schwarz inequality, we have
 \begin{equation*}
  \int_{t_0}^{t} e^{\lambda (s-t_0)} \big|\sigma(Y^{\epsilon,v_{\epsilon}}_{s;t_0,\xi^{\epsilon}}) \big|^{2} \d s \leq 2 \int_{t_0}^{t} e^{\lambda (s-t_0)} |\sigma(0) |^{2} \d s + 2 \lambda_3 \int_{t_0}^{t} e^{\lambda (s-t_0)} \int_{-\infty}^{0} \big|Y^{\epsilon,v_{\epsilon}}_{s;t_0,\xi^{\epsilon}}(\tau) \big|^{2} \mu_2(\d \tau) \d s.
 \end{equation*}
 Recall the definition of $\|\cdot \|_r$ and Proposition \ref{prop: cr} ii), we get 
 \begin{equation*}
  \int_{-\infty}^{0} \big|Y^{\epsilon,v_{\epsilon}}_{s;t_0,\xi^{\epsilon}}(\tau) \big|^{2} \, \mu_2(\d \tau) \leq \big\|Y^{\epsilon,v_{\epsilon}}_{t;t_0,\xi^{\epsilon}} \big\|_{r}^2 \int_{-\infty}^{0} e^{2r(t-s-\tau)} \, \mu_2(\d \tau) \leq e^{2r(t-s)} \, \mu_2^{(2r)} \big\|Y^{\epsilon,v_{\epsilon}}_{t;t_0,\xi^{\epsilon}} \big\|_{r}^2.
 \end{equation*}
 Applying the inequality $|a+b|^{q}\leq |a|^q+|b|^q$ for $q \in (0,1)$, we obtain
 \begin{align*}
  & \E \bigg[ \Big| \int_{t_0}^{t} e^{\lambda (s-t_0)} \, \big|\sigma(Y^{\epsilon,v_{\epsilon}}_{s;t_0,\xi^{\epsilon}}) \big|^{2} \, \d s \Big|^{3/4} \bigg]^{2/3} \\
  & \leq \frac{\sqrt{2} |\sigma(0)|}{\sqrt{\lambda}} \big( e^{\lambda (t-t_0)}-1 \big)^{1/2}+ \sqrt{\frac{2 \lambda_3 \mu_2^{(2r)}}{2r-\lambda}} \big( e^{2r(t-t_0)}-e^{\lambda(t-t_0)} \big)^{1/2} \E \Big[\big\|Y^{\epsilon,v_{\epsilon}}_{t;t_0,\xi^{\epsilon}} \big\|_{r}^{3/2}\Big]^{2/3}.
 \end{align*}
 Applying Proposition \ref{prop: cr} iii) and combining the above estimates, we obtain
 \begin{align*}
  & \E \| Y_{t;t_0,\xi^{\epsilon}}^{\epsilon,v_{\epsilon}} - Y_{t;t_0,\xi^{\epsilon}}^{0,v_{\epsilon}} \|_{r} \leq e^{-\lambda (t-t_0)/2} \E \sup_{t_0 \leq s \leq t} e^{\lambda (s-t_0)/2} |\eta^{\epsilon}(s)| \\
  & \quad \leq \sqrt{\epsilon} \, (3 \pi +1)^{2/3} e^{\frac{M}{2}} \Big( \frac{\sqrt{2} |\sigma(0)|}{\sqrt{\lambda}}+ \sqrt{\frac{2 \lambda_3 \mu_2^{(2r)}}{2r-\lambda}} e^{(2r-\lambda)(t-t_0)/2} \E \Big[\big\|Y^{\epsilon,v_{\epsilon}}_{t;t_0,\xi^{\epsilon}} \big\|_{r}^{3/2}\Big]^{2/3} \Big),
 \end{align*}
 where in the last step we used the fact $v_{\epsilon} \in \cA_{M}$. Using the uniform bounded estimate \eqref{eq: estimate-controlled-solution-2}, we know $Y_{t;t_0,\xi^{\epsilon}}^{\epsilon,v_{\epsilon}}-Y_{t;t_0,\xi^{\epsilon}}^{0,v_{\epsilon}}$ converge to zero in distribution. As the discussions at the beginning of the proof, the proof is complete.
\end{proof}

\subsection{Large Deviations Principle for stationary solution} \label{subsec: LDP for stationary solution}

In this subsection, we will prove the LDP for the stationary solution $\{\mathcal{Y}^{\epsilon}\}_{\epsilon>0}$ for crude cocycle $(U^{\epsilon},\theta)$ by proving its equivalence Laplace principle. In Theorem \ref{thm: stationary solution}, we have proved that for each $\epsilon>0$, there is a Borel measurable map $\mathcal{G}^{\epsilon}: \W \rightarrow \sC_r$ such that $\P$-a.s. $\mathcal{G}^{\epsilon}(W)(t)=\mathcal{Y}^{\epsilon}(t)$ for any $t \in \R$. Recall the measurable map $\mathcal{G}^{0}: \dot{H}^{1}(\R;\R^m) \rightarrow \sC_r$ defined in Section \ref{sec: Preliminary}. Lemma \ref{lem: variational representation} gives a variational representation for a certain functional of double-side Wiener process $W$ in $\R$, which is like \cite[Theorem 3.1]{Boue_variational_1998}. Thanks to these results, following \cite[Theorem 4.4]{Budhiraja_variational_2000}, we know that the following two conditions are sufficient for the Laplace principle of $\{\cY^{\epsilon}\}_{\epsilon>0}$. 
\begin{enumerate}
 \item [(B1)] Let $M >0$ and $\{v_{\epsilon}\}_{\epsilon>0} \subset{\cA}_{M}$. If $v_{\epsilon}$ converge to in distribution $v$ as $S_M$-valued random elements, then
 \begin{equation*}
  \mathcal{G}^{\epsilon} \Big(W(\cdot)+\frac{1}{\sqrt{\epsilon}} \int_{0}^{\cdot} v_{\epsilon}(s) \, \d s\Big) \rightarrow \mathcal{G}^{0} \Big(\int_{0}^{\cdot} v(s) \, \d s\Big) 
 \end{equation*}
 in distribution as $\epsilon \rightarrow 0$. 
 \item [(B2)] For each $M>0$, the set $\Gamma_{M,-\infty}:= \big\{ \mathcal{G}^{0} \big(\int_{0}^{\cdot} v(s) \d s \big);\, v \in S_M \big\}$ is compact in $\sC_r$.
\end{enumerate}


The condition (B2) has been proved in Lemma \ref{lem: proposition-skeleton-equation-1}. For each $v \in \cA_{M}$, we can define 
\begin{equation*}
 \cY^{\epsilon,v,\ast}:=\mathcal{G}^{\epsilon} \Big(W(\cdot)+\frac{1}{\sqrt{\epsilon}} \int_{0}^{\cdot} v (s) \, \d s\Big); \quad \cY^{0,v,\ast} := \mathcal{G}^{0} \Big(\int_{0}^{\cdot} v(s) \, \d s\Big).
\end{equation*}
Theorem \ref{thm: stationary solution} shows that $\mathcal{G}^{\epsilon}_{-n,\xi}(W(\cdot))$ converges to $\cY^{\epsilon}=\mathcal{G}^{\epsilon}(W(\cdot))$ in $L^2(\Omega, \sC_r)$ when $n \rightarrow +\infty$. By Girsanov's theorem and the fact $v \in \cA_M$, the probability measure $\P^{v,\epsilon}$ given by
\begin{equation*}
 \frac{\d \P^{v,\epsilon}}{\d \P} \Big|_{\mathcal{F}_t} = \exp \Big\{ \frac{1}{\sqrt{\epsilon}} \, \int_{-\infty}^{t} v(s) \, \d W(s) -\frac{1}{2 \epsilon} \int_{-\infty}^{t} |v(s)|^2 \, \d s\Big\} 
\end{equation*}
is equivalent with probability measure $\P$. Then, we obtain $\P$-a.s.,
\begin{equation*}
 \cY^{\epsilon,v,\ast} =\mathcal{G}^{\epsilon} \Big(W(\cdot)+\frac{1}{\sqrt{\epsilon}} \int_{0}^{\cdot} v(s) \, \d s\Big) = \lim_{n \rightarrow +\infty} \mathcal{G}^{\epsilon}_{-n,\xi} \Big(W(\cdot)+\frac{1}{\sqrt{\epsilon}} \int_{0}^{\cdot} v(s) \, \d s\Big). 
\end{equation*}
Lemma \ref{lem: truncated-controlled-equation} ii) shows that for each $t \in \R$, $Y^{\epsilon,v}_{t;-n,\xi}=\mathcal{G}^{\epsilon}_{-n,\xi} \big(W(\cdot)+\frac{1}{\sqrt{\epsilon}} \int_{0}^{\cdot} v(s) \, \d s\big)(t)$ is a Cauchy sequence in $L^1(\Omega, C_r)$. Let $\mathcal{Y}^{\epsilon,v_{\epsilon}}_{-n,\xi}(t):= Y^{\epsilon,v_{\epsilon}}_{t;-n,\xi}$ and recall the definition of $d_r$, we know
\begin{equation*}
 \P\mbox{-}a.s., \quad \cY^{\epsilon,v,\ast} =\mathcal{G}^{\epsilon} \Big(W(\cdot)+\frac{1}{\sqrt{\epsilon}} \int_{0}^{\cdot} v(s) \, \d s\Big) = \lim_{n \rightarrow +\infty} \mathcal{Y}^{\epsilon,v}_{-n,\xi}. 
\end{equation*}
There is a similar result for $\epsilon=0$ given in Lemma \ref{lem: well-posedness-skeleton-equation}. Thus, to prove the condition (B1), we only need the condition (A1) proved in Section \ref{subsec: uniform LDP} and the following uniform estimates.

\begin{enumerate}
 \item [(B-1.1)] For each $\{v_{\epsilon}\}_{\epsilon>0} \subset{\cA}_{M}$ and $\xi \in C_r$, let $\mathcal{Y}^{\epsilon,v_{\epsilon}}_{-n,\xi}(t):= Y^{\epsilon,v_{\epsilon}}_{t;-n,\xi}$, where $Y^{\epsilon,v_{\epsilon}}_{\cdot;-n,\xi}$ is the solution of Eq. \eqref{eq: truncated-controlled-equation}. Then $\{\mathcal{Y}^{\epsilon,v_{\epsilon}}_{-n,\xi}\}_{n \geq 1}\subset L^1(\Omega,\sC_r)$ is a Cauchy sequence and
 \begin{equation*}
  \lim_{n \rightarrow +\infty} \sup_{\epsilon} \E \, d_{r} \Big( \mathcal{G}^{\epsilon} \Big(W(\cdot)+\frac{1}{\sqrt{\epsilon}} \int_{0}^{\cdot} v_{\epsilon}(s) \, \d s\Big),\mathcal{Y}^{\epsilon,v_{\epsilon}}_{-n,\xi} \Big) =0. 
 \end{equation*} 
 \item [(B-1.2)] For each $v \in S_M$ and $\xi \in C_r$, it holds
 \begin{equation*}
  \lim_{n \rightarrow +\infty} \, d_{r} \Big( \mathcal{G}^{0} \Big(\int_{0}^{\cdot} v(s) \, \d s\Big) , \mathcal{G}_{-n,\xi}^{0} \Big(\int_{0}^{\cdot} v(s) \, \d s\Big) \Big)=0. 
 \end{equation*}
\end{enumerate}
The conditions (B-1.1) and (B-1.2) have been partially verified in Lemma \ref{lem: truncated-controlled-equation} and Lemma \ref{lem: well-posedness-skeleton-equation}, respectively. Now, we are ready to finalize its full proof in the following lemma.

\begin{lemma} \label{lem: uniform-controlled-equation}
 The conditions (B-1.1) and (B-1.2) hold for $\epsilon>0$ small enough.
\end{lemma}
\begin{proof}
 Lemma \ref{lem: well-posedness-skeleton-equation} and Lemma \ref{lem: truncated-controlled-equation} show that for each $v \in S_M$ and $\{v_{\epsilon}\}_{\epsilon > 0} \subset \cA_{M}$,
 \begin{align*}
  \sup_{n\geq 1} \big\|Y^{0,v}_{t;-n,\xi} \big\|_{r}^{2} & \leq 2 \, e^{2 M} \, \| \xi \|_{r}^2 + e^{2 M} \, C(\lambda); \\
  \sup_{n \geq 1} \sup_{0<\epsilon<\epsilon_0} \E \big[\| Y^{\epsilon,v_{\epsilon}}_{t;-n,\xi}\|_{r}^{3/2} \big]^{2/3} & \leq 2 \Tilde{C}_1(\epsilon_0,\lambda) \, e^{\frac{M}{2(1-\epsilon_2)}} \, \|\xi\|_{r} + 2\Tilde{C}_2(\epsilon_0,\lambda) \, e^{\frac{M}{2(1-\epsilon_2)}},
 \end{align*}
 for all $t \geq 0$, where $\epsilon_0,\lambda$ and constants $C(\lambda), \Tilde{C}_1(\epsilon_0,\lambda),\Tilde{C}_2(\epsilon_0,\lambda) $ are same as those in Lemma \ref{lem: well-posedness-skeleton-equation} and Lemma \ref{lem: truncated-controlled-equation}. Recall the definition of $d_{r}$, using the above uniform boundedness and the triangle inequality, to obtain the conditions (B-1.1) and (B-1.2), we only need to prove
 \begin{align}
  \big\|Y^{0,v}_{t;-n,\xi}- \mathcal{Y}^{0,v,\ast}(t) \big\|_{r}^{2} & \leq C \, e^{-\lambda (t+n)+ 3 M}; \nonumber \\
  \sup_{0<\epsilon<\epsilon_0} \E \, \big\| Y^{\epsilon,v_{\epsilon}}_{t;-n,\xi}- \mathcal{Y}^{\epsilon,v_{\epsilon},\ast}(t) \big\|_r & \leq C \, e^{\frac{-\lambda (t+n)}{2}+ (1+\frac{1}{1-\epsilon_2}) \frac{M}{2}}, \label{eq: estimate-controlled-solution-cauchy-0}
 \end{align}
 for any $t \geq -n$ and $n \in \N$, where constant $C>0$ only depends on the above uniform bound of $\{Y^{0,v}_{t;-n,\xi}\}_{t,n,\xi}$ and $\{ Y^{\epsilon,v_\epsilon}_{t;-n,\xi} \}_{t,n,\xi,\epsilon}$. The above first inequality comes from estimate \eqref{eq: estimate-skeleton-2} and Eq. \eqref{eq: skeleton equation 0}. The proof of the above second inequality is similar to Theorem \ref{thm: stationary solution}.
 
 For any $m>n$, $t \geq 0$ and $\tau \in (-n-t,0)$, $Y^{\epsilon,v_\epsilon}_{t;-m,\xi}(\tau)$ and $Y^{\epsilon,v_{\epsilon}}_{t;-n,Y^{\epsilon,v_{\epsilon}}_{-n;-m,\xi}}(\tau)$ satisfy the same equation, then the uniqueness of \eqref{eq: truncated-controlled-equation} shows that $Y^{\epsilon,v_\epsilon}_{t;-m,\xi}=Y^{\epsilon,v_{\epsilon}}_{t;-n,Y^{\epsilon,v_{\epsilon}}_{-n;-m,\xi}}$ almost surely. Due to $Y^{\epsilon,v_{\epsilon}}_{-n;-m,\xi}$ is $\cF_{-n}$-measurable, using Lemma \ref{lem: truncated-controlled-equation} ii), we have
 \begin{align*}
  \sup_{0<\epsilon<\epsilon_0} \E \big\| Y^{\epsilon,v_\epsilon}_{t;-m,\xi}- Y^{\epsilon,v_{\epsilon}}_{t;-n,\xi} \big\|_{r} & = \sup_{0<\epsilon<\epsilon_0} \E \big\| Y^{\epsilon,v_{\epsilon}}_{t;-n,Y^{\epsilon,v_{\epsilon}}_{-n;-m,\xi}}-Y^{\epsilon,v_{\epsilon}}_{t;-n,\xi} \big\|_{r} \\ 
  & \leq \Tilde{C}_3(\epsilon_0,\lambda) \, \sup_{0<\epsilon<\epsilon_0} \E \Big[ \big\| Y^{\epsilon,v_{\epsilon}}_{-n;-m,\xi}-\xi \big\|_{r}^{3/2} \Big]^{2/3} \, e^{\frac{-\lambda (t+n)+M}{2}}. 
 \end{align*}
 By the uniform boundedness of $\{ Y^{\epsilon,v_{\epsilon}}_{t;-n,\xi} \}_{t,n,\xi,\epsilon}$ in $L^{3/2} (\Omega, C_r)$, for any $m>n$, we obtain
 \begin{equation} \label{eq: estimate-controlled-solution-cauchy}
  \sup_{0<\epsilon<\epsilon_0} \E \big\| Y^{\epsilon,v_\epsilon}_{t;-m,\xi}- Y^{\epsilon,v_{\epsilon}}_{t;-n,\xi} \big\|_{r} \leq C \, e^{\frac{-\lambda (t+n)}{2}+\frac{2 - \epsilon_2}{1-\epsilon_2} \frac{M}{2}},
 \end{equation}
 where constant $C>0$ is independent of $m$. Thus, $\{ \mathcal{Y}^{\epsilon,v_{\epsilon}}_{-n,\xi}\}_{n \geq 1}$ is a Cauchy sequence in $L^{1} (\Omega, \sC_r)$. As mentioned in the beginning of this subsection, the limit of $\{ \mathcal{Y}^{\epsilon,v_{\epsilon}}_{-n,\xi}\}_{n \geq 1}$ in $L^1(\Omega, C_r)$ is exactly equal to $\mathcal{Y}^{\epsilon,v_{\epsilon},\ast}$. Hence $\{\mathcal{Y}^{\epsilon,v_{\epsilon},\ast}(t)\}_{t,\epsilon}$ is also uniform bounded in $L^1(\Omega, C_r)$ for all $t \in \R$. Taking $m \rightarrow +\infty$ in \eqref{eq: estimate-controlled-solution-cauchy}, we get estimate \eqref{eq: estimate-controlled-solution-cauchy-0}. The proof is completed.
\end{proof}

Now that all preparations have been made, it is time to prove Theorem \ref{thm: LDP-stationary solution}.

\begin{proof}[\bf Proof of Theorem \ref{thm: LDP-stationary solution}]
 Following \cite[Theorem 4.4]{Budhiraja_variational_2000}, we only need to prove the conditions (B1) and (B2). The condition (B2) has been proved in Lemma \ref{lem: proposition-skeleton-equation-1}. As mentioned above, we will use the condition (A1) proved in Section \ref{subsec: uniform LDP} and Lemma \ref{lem: uniform-controlled-equation} to obtain the condition (B2). 
 
 Let $M >0$ and $\{v_{\epsilon}\}_{\epsilon>0} \subset{\cA}_{M}$. If $v_{\epsilon}$ converge to in distribution $v$ as $S_M$-valued random elements, then for any bounded Lipschitz continuous functions $h: \sC_r \rightarrow \R$, we have
 \begin{align*}
  & \E |h(\cY^{\epsilon,v_{\epsilon},\ast})- h(\cY^{0,v,\ast})| \\
  & \leq \E | h(\cY^{\epsilon,v_{\epsilon},\ast})- h(\cY^{\epsilon,v_{\epsilon}}_{-n,\xi})| + \E | h(\cY^{\epsilon,v_{\epsilon}}_{-n,\xi})- h(\cY^{0,v}_{-n,\xi})| +\E | h(\cY^{0,v}_{-n,\xi})- h(\cY^{0,v,\ast})| \\
  & \leq L_{h} \Big( \sup_{0<\epsilon<\epsilon_0} \E \, d_{r} \big( \cY^{\epsilon,v_{\epsilon},\ast}, \cY^{\epsilon,v_{\epsilon}}_{-n,\xi} \big) + \E \, d_{r} \big(\cY^{0,v}_{-n,\xi}, \cY^{0,v,\ast} \big) \Big) + \E | h(\cY^{\epsilon,v_{\epsilon}}_{-n,\xi})- h(\cY^{0,v}_{-n,\xi})|,
 \end{align*}
 where $ \cY^{0,v}_{-n,\xi}(\cdot):=Y^{0,v}_{\cdot; -n \xi} $ and $\cY^{\epsilon,v_{\epsilon}}_{-n,\xi}(\cdot):=Y^{\epsilon,v_{\epsilon}}_{\cdot;-n,\xi} $ are the solution map of Eqs. \eqref{eq: skeleton equation} and \eqref{eq: truncated-controlled-equation}, respectively, and $L_h$ is the Lipschitz coefficient of $h$. The condition (A1) yields
 \begin{equation*}
  \limsup_{\epsilon \rightarrow 0}\E | h(\cY^{\epsilon,v_{\epsilon},\ast})- h(\cY^{0,v,\ast})| \leq L_{h} \Big( \sup_{0<\epsilon<\epsilon_0} \E \, d_{r} \big( \cY^{\epsilon,v_{\epsilon},\ast}, \cY^{\epsilon,v_{\epsilon}}_{-n,\xi} \big) + \E \, d_{r} \big(\cY^{0,v}_{-n,\xi}, \cY^{0,v,\ast} \big) \Big).
 \end{equation*}
 Using (B1-1) and (B1-2) and taking $n \rightarrow +\infty$, we obtain 
 \begin{equation*}
  \limsup_{\epsilon \rightarrow 0}\E | h(\cY^{\epsilon,v_{\epsilon},\ast})- h(\cY^{0,v,\ast})|=0.
 \end{equation*}
 Thus, $\cY^{\epsilon,v_{\epsilon},\ast}$ converge in distribution to $\cY^{0,v,\ast}$ as $\epsilon \rightarrow 0$, and the condition (B1) is verified. Following \cite[Theorem 4.4]{Budhiraja_variational_2000}, $\{ \cY^{\epsilon}\}_{\epsilon>0}$ satisfy the Laplace principle with rate function $\cI_{-\infty}$
\end{proof}

\appendix

\section{Analysis tool} \label{sec: analysis tools}

In this section, we will show some useful analysis tools. We first introduce a number of properties of the Polish space $C_r$, see \cite{Hino_Functionaldifferential_1991} for details.
\begin{proposition}\label{prop: cr}
 The space $C_{r}$ has the following properties.
 \begin{itemize}
  \item [i)] The space $C_{r}$ is complete and separable. If $\{\phi^{n}\}$ is a Cauchy sequence in $C_{r}$, then exists a uniqueness $\phi\in C_{r}$ such that $||\phi^{n}-\phi||_{r} \rightarrow 0$ as $n\to\infty$. As a consequence, $\{\phi^{n}(\tau)\}_{n}$ converges to $\phi(\tau)$ uniformly in all compact subset $K \in (-\infty,0]$.
  \item [ii)] For fixed $T \in \R$ and function $\varphi \in C((-\infty,T])$, for each $t \in (-\infty,T)$, we can define
   \begin{equation*}
    \varphi_{t}(\tau) := \varphi(t+\tau) , \quad \forall \, \tau \in (-\infty,0].
   \end{equation*}
   If $\varphi_{T} \in C_{r}$, then $\varphi_t $ is a $C_r$-valued continuous function and 
   \begin{equation*}
    \| \varphi_{t} \|_{r}^2 \leq e^{2r (T-t)} \, \| \varphi_{T} \|_{r}^2, \quad \forall \, t \leq T.
   \end{equation*}
  \item [iii)] Let $\varphi \in C((-\infty,T])$ and $-\infty<t_0<t\leq T$. If $\varphi_{t_0} \in C_r$ and $0<\lambda \leq 2r$, then 
   \begin{equation*}
    \| \varphi_t \|_{r}^2 \leq e^{-2r (t-t_0)} \, \| \varphi_{t_0}\|_r^2 + e^{-\lambda t} \, \sup_{t_0 \leq s \leq t} e^{\lambda s} |\varphi(s)|^2.
   \end{equation*}
\end{itemize}
\end{proposition}

The following Lemma can be found in the proof of \cite[Theorem 3.2]{Wu_Stochastic_2017}.
\begin{lemma} \label{lem: estimate for mu}
 For any $0<\lambda<2r$, $t \geq 0$ and $\phi_{t} := \left\{ \phi(t+\tau); \, \tau \in (-\infty,0] \right\} \in C_r$, it holds
 \begin{equation*}
  \int_{0}^{t}\int_{-\infty}^{0}e^{\lambda s} \vert \phi(s+\tau) \vert^{2} \, \mu_{j}(\d \tau) \, \d s \leq \frac{\mu_{j}^{(2r)}}{2r-\lambda} \|\phi_{0}\|_{r}^{2} + \mu_{j}^{(2r)} \int_{0}^{t} e^{\lambda s} \vert \phi(s)|^{2} \, \d s, 
 \end{equation*}
 for each $ {j}=1,2$, where $\mu_j^{(2r)}$ defined as \eqref{def: mu}.
\end{lemma}

The following stochastic Gr\"onwall inequality has been proved in \cite{Scheutzow_stochastic_2013}.

\begin{lemma}\label{lem: Stochastic Gronwall}
  Let $Z$ and $H$ be nonnegative, adapted processes with continuous paths and assume that $\psi$ is nonnegative and progressively measurable. Let $M$ be a continuous local martingale starting at $0$. If the following condition holds for all $t\geq 0$,
  \begin{equation*}
  Z(t)\leq\int_{0}^{t}\psi(s) \, Z(s)\,\mathrm{d}s+M(t)+H(t), 
  \end{equation*}
  then for $p\in(0,1)$, $p_1 \in (1,1/p)$ and $p_2= \frac{p_1}{p_1-1}$, we have
  \begin{equation*}
  \E \sup_{0 \leq s \leq t} Z^{p}(s)\leq \Big( (c_{p} p_1 +1)\, \E\big[\sup_{0 \leq s \leq t} H(s) \big]^{p p_1}\Big)^{1/p_1} \left(\E \exp \Big\{p p_2 \int_{0}^{t}\psi(s)\,\d s \Big\}\right)^{1/p2} ,
  \end{equation*}
  where $c_p := (4 \wedge \frac{1}{p}) \frac{\pi p}{\sin(\pi p)}>0$.
\end{lemma}

Kurtz in \cite[Theorem 3.14]{Kurtz_YamadaWatanabeEngelbert_2007} shows the general Yamada–Watanabe–Engelbert Theorem for the compatible solution. Following this result and \cite{Rockner_YamadaWatanabe_2008}, we can get the similar result for $\cF_{t}$-adapted solution to \eqref{eq: truncated SFDEs epsilon} and \eqref{eq: stationary solution}, which are driven by double-side Wiener process $W$.
\begin{lemma}\label{lem: Yamada-Watanabe}
 Let $S=\W $ (or $ \W \times C_r$) and $\mathcal{P}_{W}(\sC_{r} \times S)$ be the Borel probability space on $\sC_r \times S$ satisfying the marginal distribution on $S$ is $\P \circ W^{-1}$ (or $ \P \circ (W,\xi)^{-1}$). Let $S_{\Gamma} \subset \mathcal{P}_{W}(\sC_r \times S)$ denote the collection of $\cF_{t}$-adapted solution measures with the constraint $\Gamma$. If $S_{\Gamma} $ is a convex non-empty set, then the following are equivalent: 
 \begin{itemize}
  \item [i)] Pathwise uniqueness holds for $\cF_{t}$-adapted weak solutions.
  \item [ii)] Joint uniqueness in law holds, and there exists a $\cF_{t}$-adapted strong solution.
 \end{itemize}
\end{lemma}
\begin{proof}
 ii) implies i) is immediate, so we focus on obtaining ii) from i). 
  
 Let $Y=W$ (or $(W, \xi)$). Let $\zeta_1 $ and $\zeta_2$ uniformly distributed on $[0,1]$ and $\zeta_1$, $\zeta_2$ and $Y$ are independent. Then by \cite[Lemma A.1]{Kurtz_YamadaWatanabeEngelbert_2007}, for any $\mu_1,\mu_2 \in S_{\Gamma}$, there exists map $F_1: S \times [0,1]$ and $F_2: S \times [0,1]$ such that $(F_{1}(Y, \zeta_1),Y)$ and $(F_{2}(Y, \zeta_2),Y)$ has the distribution $\mu_1$ and $\mu_2$, respectively. Let $\mathcal{B}_{t}(\sC_r)$ be the $\sigma$-algebra generated by all maps $\pi_{s}: \sC_r \rightarrow C_r$, $s \leq t$, where $\pi_{s} z := z(s)$, $z \in \sC_r$. Next, we will show for any $\mathcal{B}_{t}(\sC_r)$ measurable map $f : \sC_r \rightarrow \R$,
 \begin{equation*}
  (S,[0,1]) \ni (y,s) \mapsto f(F_{1}(y, s))= \int_{\sC_r} f(z) K_{F_1}(y,s, \d z)
 \end{equation*}
 is $\mathcal{B}_{t}(\W) \otimes \mathcal{B}([0,1])$ (defined similar as $\mathcal{B}_{t}(\sC_r)$) measurable, where $K_{F_1}(y,s,z)$ is a transition kernel function. It is sufficient to prove that for any $B \in \mathcal{B}_{t}(\sC_r)$, $A_1 \in \mathcal{B}_{t}(\W)$ and 
 \begin{align*}
  A_2 \in \sigma \big( & \{ w(t_1)-w(t) \in D_1, w(t_2)-w(t_1) \in D_2, \cdots, w(t_k)-w(t_{k-1}) \in D_k \big. \\
  & \quad \big. : \, t \leq t_1 < t_2 < \cdots <t_k, \, D_1,D_2, \cdots D_k \in \mathcal{B}(\R^m), \, w \in \W \} \big)
 \end{align*}
 the following identity holds
 \begin{align} \label{eq: adapted}
  & \int_{S \times [0,1]} \mathrm{1}_{ A_1 \cap A_2}(\Pi y) \, K_{F_1}(y,s,B) \, \P \circ Y^{-1} (\d y) \, \d s \nonumber \\
  &= \int_{S \times [0,1]} \mathrm{1}_{ A_1 \cap A_2}(\Pi y) \, \E [K_{F_1}(\cdot,B) | \mathcal{B}_t(\W) \otimes \mathcal{B}([0,1])] \, \P \circ Y^{-1} (\d y) \, \d s,
 \end{align}
 where $\Pi: Y \mapsto W$, since all $A_1 \cap A_2$ generate $\cF$. The left-hand side of \eqref{eq: adapted} is equal to
 \begin{equation*}
  \int_{\Omega} \mathrm{1}_{ A_1 \cap A_2}(W) \mathrm{1}_{B}(Z) \d \P= \P \circ W^{-1} (A_2) \int_{\Omega} \mathrm{1}_{ A_1 }(W) \mathrm{1}_{B}(Z) \d \P,
 \end{equation*}
 where $(Z,W)$ is a $\cF_t$-adapted weak solution and the last step used $\mathrm{1}_{A_2}(W)$ is independent of $\mathcal{F}_{t}$. Due to $A_2$ is independent $\mathcal{B}_{t}(\W)$, the right-hand side of \eqref{eq: adapted} is equal to
 \begin{align*}
  & \int_{S } \mathrm{1}_{ A_2}(\Pi y) \, \P \circ Y^{-1} (\d y) \times \int_{S \times [0,1]} \mathrm{1}_{ A_1 }(\Pi y) \, \E [K_{F_1}(\cdot,B) | \mathcal{B}_t(\W) \otimes \mathcal{B}([0,1])] \, \P \circ Y^{-1} (\d y) \, \d s \\
  & \quad = \P \circ W^{-1} (A_2) \times \int_{S \times [0,1]} \mathrm{1}_{ A_1 }(\Pi y) \, K_{F_1}(\cdot,B) \, \P \circ Y^{-1} (\d y) \, \d s \\
  & \quad = \P \circ W^{-1} (A_2) \int_{\Omega} \mathrm{1}_{ A_1 }(W) \mathrm{1}_{B}(Z) \d \P.
 \end{align*} 
 Thus, \eqref{eq: adapted} holds, and furthermore $f(F_1(Y,\zeta_1))$ is $\mathcal{F}_{t} \otimes \sigma(\zeta_1)$ measurable. The similar result holds for $F_2(Y,\zeta_2)$.
 By pathwise uniqueness and \cite[Lemma A.2]{Kurtz_YamadaWatanabeEngelbert_2007}, there exists Borel measurable $F : S \rightarrow \sC_r$ such that $\P$-a.s. $F(Y)=F_{1}(Y, \zeta_1)=F_{2}(Y, \zeta_2)$, then $F(Y)$ is the strong solution. By the measurability properties of $F_1(Y,\zeta_1)$, we obtain that $F(Y)$ is $\mathcal{F}_{t}$-adapted.
\end{proof}

Lehec in \cite[Theorem 9]{Lehec_Representation_2013} gives a variational representation for a certain functional of Wiener process in $[0,+\infty)$, which is a generalized version of Bou\'e-Dupuis formula \cite[Theorem 3.1]{Boue_variational_1998}. By the same way as \cite[Theorem 9]{Lehec_Representation_2013}, we have the following variational representation for a certain functional of double-side Wiener process $W$ in $(-\infty,+\infty)$. 

\begin{lemma}[The variational representation.] \label{lem: variational representation}
 Let $f$ be a bounded Borel measurable function mapping $\W$ into $\R$. Then
 \begin{equation*}
  - \log \E \, e^{-f (W)} = \inf_{v \in \mathcal{A}} \E \Big\{ \frac{1}{2} \int_{-\infty}^{+\infty} |v(s)|^2 \, \d s + f \big( W+ \int_{0}^{\cdot} v(s) \, \d s \big) \Big\},
 \end{equation*}
 where $\mathcal{A}$ is defined in \eqref{def: variational space}.
\end{lemma}
\begin{proof}
 The proof is similar to \cite[Theorem 9]{Lehec_Representation_2013}. Let $\gamma $ be the Wiener measure and probability measure $\upsilon \ll \gamma$. By Donsker-Varadhan variational formula of relative entropy \cite[Lemma 1.4.3]{Dupuis_Weak_1997}, we have the Legendre duality:
 \begin{equation*}
  H(\upsilon| \gamma) := \int_{\Omega} \frac{\d \upsilon}{ \d \gamma} \log \big( \frac{\d \upsilon}{ \d \gamma} \big) \, \d \gamma = \sup_{f} \Big( \int_{\Omega} f \d \upsilon -\log \big( \int_{\Omega} e^{f} \d \gamma \big) \Big),
 \end{equation*}
 where the maximum is taken over all bounded Borel measurable function $f: \W \rightarrow \R$.
 
 Using the Legendre duality and Girsanov theorem, we obtain the corresponding version of \cite[Proposition 1 and Lemma 8]{Lehec_Representation_2013}. This allows us to transform the problem into proving the corresponding version of \cite[Theorem 7]{Lehec_Representation_2013}. Specifically, for all $\upsilon \in S$, we need to prove
\begin{equation} \label{relative entropy}
 H(\upsilon| \gamma) = \sup_{f} \Big( \int_{\Omega} f \d \upsilon -\log \big( \int_{\Omega} e^{f} \d \gamma \big) \Big)= \min_{v} \Big( \frac{1}{2} \int_{-\infty}^{+\infty} |v(s)|^2 \, \d s \Big),
\end{equation}
 where the minimum is taken over all $v \in \mathcal{A}$ such that $W+ \int_{0}^{\cdot} v(s) \, \d s$ has law $\upsilon$. The set $S$ consists of probability measures on $(\W, \mathcal{B}(\W), \gamma )$ that possess a density of the form $F (\omega) = \Phi(\omega_{t_1}, \cdots, \omega_{t_n} )$, for some $n \in \N$, $-\infty< t_1<t_2<\cdots<t_n<+\infty$ and some function $\Phi: (\R^{d})^{n} \rightarrow \R $ satisfying Lipschitz continuity for both $\Phi$ and $\nabla \Phi$, and $\Phi \geq \epsilon$ for some $\epsilon>0$. It is noteworthy that $\upsilon \in S$ only depends on finite moments in time, then the identity \eqref{relative entropy} holds as established in \cite[Theorem 7]{Lehec_Representation_2013}. We obtain the desired result.
\end{proof}

\section*{Acknowledgements}

Y. Liu is supported by the National Natural Science Foundation of China (No. 12231002) and Center for Statistical Science, PKU.

\printbibliography

\end{document}